\documentclass[11pt]{amsart} \textwidth=14.5cm \oddsidemargin=1cm
\usepackage{amsmath,wasysym}
\usepackage{amsxtra}
\usepackage{amscd}
\usepackage{amsthm}
\usepackage{amsfonts}
\usepackage{amssymb}
\usepackage{eucal}
\usepackage{graphics, color}
\usepackage{hyperref,mathrsfs}

\input prepictex
\input pictex
\input postpictex
\newtheorem{thm}{Theorem}[section]
\newtheorem{lem}[thm]{Lemma}
\newtheorem{cor}[thm]{Corollary}
\newtheorem{prop}[thm]{Proposition}

\theoremstyle{definition}

\theoremstyle{remark}
\newtheorem{rem}[thm]{Remark}

\numberwithin{equation}{section}

\begin{document}

%Referring commands:
\newcommand{\thmref}[1]{Theorem~\ref{#1}}
\newcommand{\secref}[1]{Section~\ref{#1}}
\newcommand{\lemref}[1]{Lemma~\ref{#1}}
\newcommand{\propref}[1]{Proposition~\ref{#1}}
\newcommand{\corref}[1]{Corollary~\ref{#1}}
\newcommand{\remref}[1]{Remark~\ref{#1}}
\newcommand{\eqnref}[1]{(\ref{#1})}
\newcommand{\exref}[1]{Example~\ref{#1}}

%Simplified symbols:
\newcommand{\nc}{\newcommand}
\nc{\Z}{{\mathbb Z}}
\nc{\hZ}{{\hf+\mathbb Z}}
\nc{\C}{{\mathbb C}}
\nc{\N}{{\mathbb N}}
\nc{\F}{{\mf F}}
\nc{\Q}{\ol{Q}}
\nc{\la}{\lambda}
\nc{\ep}{\varepsilon}
\nc{\h}{\mathfrak h}
\nc{\n}{\mf n}
\nc{\A}{{\mf a}}
\nc{\G}{{\mathfrak g}}
\nc{\SG}{\overline{\mathfrak g}}
\nc{\DG}{\widetilde{\mathfrak g}}
\nc{\D}{\mc D} \nc{\Li}{{\mc L}} \nc{\La}{\Lambda} \nc{\is}{{\mathbf
i}} \nc{\V}{\mf V} \nc{\bi}{\bibitem} \nc{\NS}{\mf N}
\nc{\dt}{\mathord{\hbox{${\frac{d}{d t}}$}}} \nc{\E}{\mc E}
\nc{\ba}{\tilde{\pa}} \nc{\half}{\frac{1}{2}} \nc{\mc}{\mathcal}
\nc{\mf}{\mathfrak} \nc{\hf}{\frac{1}{2}}
\nc{\hgl}{\widehat{\mathfrak{gl}}} \nc{\gl}{{\mathfrak{gl}}}
\nc{\hz}{\hf+\Z}
\nc{\dinfty}{{\infty\vert\infty}} \nc{\SLa}{\overline{\Lambda}}
\nc{\SF}{\overline{\mathfrak F}} \nc{\SP}{\overline{\mathcal P}}
\nc{\U}{\mathfrak u} \nc{\SU}{\overline{\mathfrak u}}
\nc{\ov}{\overline}
\nc{\wt}{\widetilde}
\nc{\wh}{\widehat}
\nc{\sL}{\ov{\mf{l}}}
\nc{\sP}{\ov{\mf{p}}}
\nc{\osp}{\mf{osp}}
\nc{\spo}{\mf{spo}}
\nc{\hosp}{\widehat{\mf{osp}}}
\nc{\hspo}{\widehat{\mf{spo}}}
\nc{\hh}{\widehat{\mf{h}}}
\nc{\even}{{\bar 0}}
\nc{\odd}{{\bar 1}}
\nc{\mscr}{\mathscr}

\newcommand{\blue}[1]{{\color{blue}#1}}
\newcommand{\red}[1]{{\color{red}#1}}
\newcommand{\green}[1]{{\color{green}#1}}
\newcommand{\white}[1]{{\color{white}#1}}

 \advance\headheight by 2pt

\title[Finite-dimensional representations of $\mf{q}(n)$ of half-integer weights]
{Finite-dimensional half-integer weight modules over queer Lie superalgebras}

\author[Cheng]{Shun-Jen Cheng$^\dagger$}
\thanks{$^\dagger$Partially supported by a MOST and an Academia Sinica Investigator grant}
\address{Institute of Mathematics, Academia Sinica, Taipei,
Taiwan 10617} \email{chengsj@math.sinica.edu.tw}

\author[Kwon]{Jae-Hoon Kwon$^{\dagger\dagger}$}
\thanks{$^{\dagger\dagger}$Partially supported by a NRF-grant 2011-0006735.}
\address{Department of Mathematics, Sungkyunkwan University
2066 Seobu-ro, Jangan-gu, Suwon, Korea}
\email{jaehoonkw@skku.edu}

\begin{abstract}
We give a new interpretation of representation theory of the finite-dimensional half-integer weight modules over the queer Lie superalgebra $\mf{q}(n)$. It is given in terms of Brundan's work of finite-dimensional integer weight $\mf{q}(n)$-modules by means of Lusztig's canonical basis. Using this viewpoint we compute the characters of the finite-dimensional half-integer weight irreducible modules. For a large class of irreducible modules whose highest weights are of special types (i.e., totally connected or totally disconnected) we derive closed-form character formulas that are reminiscent of Kac-Wakimoto character formula for classical Lie superalgebras.
\end{abstract}

\subjclass[2010]{17B67}

\maketitle

\section{Introduction}
The character problem for finite-dimensional irreducible modules over the basic classical Lie superalgebras has now been settled (see, e.g., \cite{Br1,CLW,CL,Ger, GS,Mar,Ser,Sv1, SZ1, SZ2, vdJ}). For the queer Lie superalgebra $\mf{q}(n)$, the solution of this problem was first given by Penkov and Serganova \cite{PS1,PS2}. %Let $\Lambda^+$ be the set of dominant weights for $\mf q(n)$.
Indeed, given dominant weights $\la, \mu$ for $\mf{q}(n)$, they provide an algorithm for computing the multiplicity of an irreducible $\mf{q}(n)$-module $L(\mu)$ with highest weight $\mu$ in the cohomology groups of certain bundle on $\Pi$-symmetric projective space associated to $\la$, and, as a consequence, they determine the coefficient $a_{\la\mu}$ of the character of $L(\mu)$ in the expansion of the character of the associated Euler characteristic $E(\la)$.

A completely different approach to computing the transition matrix $(a_{\la\mu})$ between the characters of $E(\la)$ and $L(\mu)$ in the case when $\la, \mu$ are integer dominant weights was given by Brundan \cite{Br2}. Indeed, let $U_q(\mf b_\infty)$ be the quantum group of type $B$ with infinite rank, and let $\mscr F^n$ the $n$th exterior power of its natural representation $\mscr V$ (cf.~\cite{JMO}). It was shown that $(a_{\la\mu})$ is given by the transpose of the transition matrix between the canonical basis and the natural monomial basis of $\mscr F^n$ at $q=1$. This, together with an explicit combinatorial algorithm for computing the canonical basis elements, gives a solution of the irreducible character problem in the case of integer weight modules. In the equivalent dual picture the Grothendieck group of the category of finite-dimensional integer weight modules can be realized  by a $U_q(\mf{b}_\infty)$-module $\mscr E^n$, which is dual to $\mscr F^n$ with respect to a bilinear form on $\mscr V^{\otimes n}$ relating the canonical and dual canonical basis on the ambient space. In this dual picture $E(\la)$ is mapped to a standard basis element $E_\la$, which in turn is dual to the standard monomial basis of $\mscr F^n$, and $L(\mu)$ is mapped to a dual canonical basis element $L_\la$ at $q=1$. Further discussion and description for the inverse of $(a_{\la\mu})$ are given in \cite{SZ3}

In this paper we consider the finite-dimensional half-integer weight irreducible $\mf q(n)$-modules and their characters. We show that they also can be explained nicely in terms of the Brundan's results on integer weight modules. To be more precise, we take a quotient space $\mscr E^{n, \times}$ of $\mscr E^n$ by a subspace spanned by standard basis elements $E_\la$, where $\la$ has at least one zero part considered as a  generalized partition. Then we show that the transition matrix $(a_{\la\mu})$ for half-integer dominant weights $\la$ and $\mu$ is the same as the transition matrix $(a_{\la^\natural\mu^\natural})$ between the two bases of $\mscr E^{n,\times}$ given by $\{{\bf E}_{\la^\natural}\}$ and $\{{\bf L}_{\mu^\natural}\}$ respectively (Theorem \ref{thm:main}), where $\natural$ is a natural bijection from the set of half-integer dominant weights to the set of integer dominant weights with no zero parts (see \eqref{sharp:map}). Here, the notations ${\bf E}_{\la^\natural}$ and ${\bf L}_{\mu^\natural}$ stand for the images of ${E}_{\la^\natural}$ and ${L}_{\mu^\natural}$ under the canonical projection from $\mscr E^n$ onto $\mscr E^{n,\times}$. We remark that $\mscr E^{n,\times}$ also has a well-defined bar involution induced from that of $\mscr E^n$, and $\{{\bf L}_{\mu^\natural}\}$ forms a unique bar-invariant basis of $\mscr E^{n,\times}$ satisfying the properties of the dual canonical basis, although $\mscr E^{n,\times}$ is certainly not a $U_q(\mf{b}_\infty)$-module.

In particular, we have (up to some shift of weights) the same combinatorial formula for $a_{\la\mu}$ for half-integer dominant weights as in \cite{Br2}. Furthermore, we obtain a nice closed-form formula for the inverse matrix of $(a_{\la\mu})$ and hence for the character formula for $L(\la)$ as well (Theorem \ref{thm:char for half}), which turns out to be essentially the same as the formula  for general linear Lie superalgebras given in \cite{Br1}. This is proved by establishing a correspondence from half-integer dominant weights for $\mf q(n)$ to integer dominant weights for for the general linear Lie superalgebra $\gl(p|q)$ for some $p, q$ with $p+q=n$ preserving the Bruhat orders on each weight lattice. As an interesting application, we obtain Kac-Wakimoto type character formulas for $L(\la)$ when $\la$ is either a totally connected or a totally disconnected half-integer dominant weight (Theorem \ref{thm:KW formula}).

We conclude this introduction with the organization of the paper. In Section 2, we briefly recall some preliminary background and notations. In Section 3, we compare the Bruhat order in \cite{PS2} with the ones in \cite{Br1,Br2}. In Section 4, we present the Euler characteristic in a purely algebraic way using the dual Zuckerman functor and establish some of its basic properties. In Section 5, we prove our main observation on $a_{\la\mu}$ for half-integer dominant weights by analyzing the algorithm of Penkov and Serganova. In Section 6, we realize the Grothendieck group for the finite-dimensional half-integer weight modules, and compute the irreducible characters. Finally, as an application, Kac-Wakimoto type character formulas \cite{KW} are given in Section \ref{sec:KW:form} for finite-dimensional irreducible half-integer modules whose highest weights are either totally connected or totally disconnected.  Note that, just as in the case of classical Lie superalgebras, our formulas can then be regarded as generalizations of the Bernstein-Leites character formula \cite{BL}.

\vspace{.2cm} \noindent {\bf Acknowledgment.}
We thank Dimitar Grantcharov for raising with the first author the question about possible Kac-Wakimoto type character formula for the queer Lie superalgebra, which was one of the original motivation for investigation in this paper. The second author thanks Academia Sinica in Taipei for hospitality and support, where part of this work was carried out.

\vspace{.2cm} \noindent {\bf Notation.} Let $\Z$, $\N$, and $\Z_+$ stand for the sets of all, positive, and
non-negative integers, respectively. All vector spaces, algebras,
etc., are over the complex field $\C$.

\section{Preliminaries}\label{sec:prel}

Fix a positive integer $n\ge 1$. Let $\C^{n|n}$ be the complex superspace of dimension $(n|n)$. Choose an ordered basis $\{v_{\ov{1}},\ldots,v_{\ov{n}}\}$ for the even subspace $\C^{n|0}$ and an ordered basis $\{v_1,\ldots,v_n\}$ for the odd subspace $\C^{0|n}$ so that the general linear Lie superalgebra $\gl(n|n)$ may be realized as $2n \times 2n$ complex matrices indexed by $I(n|n):=\{\,\ov{1}<\ldots<\ov{n}<1<\ldots<n\,\}$.
The subspace
\begin{equation}\label{eq:q(n)}
{\mf q}(n):=\left\{\,
\begin{pmatrix}
A & B \\
B & A
\end{pmatrix}
\,\Big\vert\,\text{$A$, $B$ : $n \times n$ matrices}\,\right\}
\end{equation}
forms a subalgebra of $\gl(n|n)$ called the {\em queer Lie superalgebra}.

Let $\G$ stand for the Lie superalgebra ${\mf q}(n)$ from now on.
We denote by $E_{ab}$ the elementary matrix in $\gl(n|n)$ with $(a,b)$-entry $1$ and other entries $0$, for $a,b\in I(n|n)$.
We have a linear basis $\{\,e_{ij},\ov{e}_{ij}\,|\,1\leq i,j\leq n\,\}$ of $\G$, where $e_{ij}=E_{\ov{i}\ov{j}}+E_{ij}$ and $\ov{e}_{ij}=E_{\ov{i}j}+E_{i\ov{j}}$.  Note that $\G_{\even}$ is spanned by $\{\,e_{ij}\,|\,1\leq i,j\leq n\,\}$, which is isomorphic to the general linear Lie algebra $\gl(n)$.

Let $\h=\h_\even\oplus\h_\odd$ be the standard Cartan subalgebra of $\G$, which consists of matrices of the form  \eqref{eq:q(n)} with $A$ and $B$ diagonal. Then $\{\,h_i:=e_{ii}\,|\,\,1\leq i\leq n\,\}$ and $\{\,\ov{h}_i:=\ov{e}_{ii}\,|\,1\leq i\leq n\,\}$ form linear bases of $\h_\even$ and $\h_\odd$, respectively.
Let $\{\,\varepsilon_i\,|\,1\leq i\leq n\,\}$ be the basis of $\h^\ast_{\even}$ dual to $\{\,h_i\,|\,1\leq i\leq n\,\}$. We define a symmetric bilinear form $(\ ,\ )$ on $\h_\even^\ast$ by $(\varepsilon_i,\varepsilon_j)=\delta_{ij}$, for $1\leq i,j\leq n$.
Let $\mf b$ be the standard Borel subalgebra of $\G$, which consists of matrices of the form  \eqref{eq:q(n)} with $A$ and $B$ upper triangular. We have ${\mf b}=\h \oplus {\mf n}$, where ${\mf n}$ is the nilradical spanned by $\{\,e_{ij},\ov{e}_{ij}\,|\,1\leq i<j\leq n\,\}$. Also, denote by $\mf n_-$ the opposite nilradical so that $\G=\mf b\oplus \mf n_-$.
We denote by $\Phi^+$ and $\Phi^-$ the sets of positive and negative roots with respect to $\h_\even$, respectively, and denote by $\Pi$ the set of simple roots of $\G_\even$. We have
$\Phi^-=-\Phi^+$,
$\Pi=\{\,\varepsilon_i-\varepsilon_{i+1}\,|\,1\leq i\leq n-1\,\}$, and $\Phi^+=\Phi^+_\even \sqcup \Phi^+_\odd$,
where $\Phi^+_{\even}$ and $\Phi^+_{\odd}$ stand for the sets of positive even and odd roots, respectively. Ignoring parity we have $\Phi^+_\even=\Phi^+_\odd =\{\,\varepsilon_i-\varepsilon_j\,|\,1\leq i<j\leq n\,\}$.

Furthermore, we denote by $W$, the Weyl group of $\G$, which is the Weyl group of the reductive Lie algebra $\G_\even$ and hence acts naturally on $\h^*_{\even}$ by permutation. We also denote by $\ell(w)$ the length of an element $w\in W$.

For a $\G$-module $V$ and $\mu\in \h^\ast_\even$, let $V_\mu=\{\,v\in V\,|\,h\cdot v=\mu(h)v \text{ for $h\in\h_\even$}\, \}$ denote its $\mu$-weight space.  For a $\G$-module $V$ with weight space decomposition $V=\bigoplus_{\mu\in\h_\even^\ast}V_\mu$, its character is defined by
$\text{ch}V:=\sum_{\mu\in\h^*_\even}\dim V_\mu e^\mu$, where $e$ is an indeterminate.

Let $\la=\sum_{i=1}^n\la_i\varepsilon_i \in \h^*_\even$ be given. We put $\ell(\la)$ to be the number of $i$'s with $\la_i\neq 0$.
Consider a symmetric bilinear form on $\h_\odd$ given by $\langle \cdot, \cdot \rangle_\la = \la([\cdot, \cdot])$, and let $\h'_\odd$ be a maximal isotropic subspace associated to $\langle \cdot, \cdot \rangle_\la$. Put $\h'=\h_\even \oplus \h'_\odd$. Let $\C v_\la$ be the one-dimensional $\h'$-module with $h\cdot v_\la = \la(h) v_\la$ and $h'\cdot v_\la=0$, for $h\in \h_\even$ and $h'\in \h'_\odd$. Then $W_\la:={\rm Ind}_{\h'}^\h \C v_\la$ is an irreducible $\h$-module of dimension $2^{\lceil \ell(\la)/2\rceil}$, where $\lceil\cdot\rceil$ denotes the ceiling function. We put $\Delta(\la)={\rm Ind}_{\mf b}^{\G}W_\la$, where $W_\la$ is extended to a ${\mf b}$-module in a trivial way, and define $L(\la)$ to be the unique irreducible quotient of $\Delta(\la)$. Note that it is an $\h_\even$-semisimple highest weight $\G$-module of highest weight $\la$. The center of $\G$ was determined by Sergeev \cite{Sv}. We denote by $\chi_\la$ the central character of $L(\la)$ (see, e.g., \cite[Section 2.3]{CW} for more details). We say that a positive root $\varepsilon_i-\varepsilon_j$ ($i<j$) is atypical to $\la$ if $\la_i+\la_j= 0$. Recall that a weight $\la$ is said to be {\em atypical} if there is a positive root atypical to $\la$, and {\em typical} otherwise \cite{Pe}. The {\em degree of atypicality of $\la$} is the maximal number of positive even roots which are mutually orthogonal and atypical to $\la$.

For $\epsilon=0, \tfrac{1}{2}$, let $\La_{\epsilon+\Z}:=\sum_{i=1}^n(\epsilon+\Z) \varepsilon_i\subseteq\h^*_{\bar 0}$ and put $\La=\La_\Z \cup\La_{\hf+\Z}$. We say that a $\G$-module $V$ is an {\em integer weight module} and {\em half-integer weight module} if it has a weight space decomposition $V=\bigoplus_{\mu\in\La_{\epsilon+\Z}}V_\mu$ for $\epsilon=0$ and $\hf$, respectively.

 Let $\mc O_{\Z}$ and $\mc O_{\hZ}$ denote the BGG categories of $\h_\even$-semisimple integer weight and half-integer weight $\G$-modules, respectively. Let $\mc O^{\texttt{fin}}$ be the category of finite-dimensional $\G$-modules. Set
\begin{align*}
\mc O:=\mc O_\Z\cup\mc O_\hZ,\quad
\mc O^{\texttt{fin}}_{\epsilon+\Z}:=\mc O_{\epsilon+\Z}\cap \mc O^{\texttt{fin}},
\end{align*}
for $\epsilon=0, \tfrac{1}{2}$.
%For a parabolic subalgebra $\mf p\subseteq \G$ denote the corresponding parabolic subcategory of $\mc O$ by $\mc O^{\mf p}$.
For a module category $\mc C$ given above, we denote by $K(\mc C)$ the corresponding Grothendieck group spanned by $[M]$ ($M\in \mc C$), where $[M]$ stands for the equivalence class of the module $M$.

Define
\begin{equation*}
\La^+_{\epsilon+\Z}:=\left\{\la=\sum_{i=1}^n\la_i\ep_i\in \La_{\epsilon+\Z}\,\Bigg\vert\,  \la_i\ge\la_{i+1},\, \la_i=\la_{i+1}\text{ implies }\la_i=0\text{ for all $i$}\,\right\},
\end{equation*}
for $\epsilon=0, \tfrac{1}{2}$, and put $\La^+=\La^+_\Z\cup \La^+_\hZ$.
We call a weight $\nu$ in $\Lambda^+$, $\Lambda^+_{\Z}$ and $\La^+_{\hf+\Z}$ {\em dominant},  {\em integer dominant}, and {\em  half-integer dominant}, respectively.
According to \cite[Theorem 4]{Pe} we have $L(\la)\in \mc{O}^{\texttt{fin}}_{\epsilon+\Z}$ if and only if $\la\in \La_{\epsilon+\Z}^+$ (see also \cite[Theorem 2.18]{CW}).
We also let
\begin{equation*}
\La_{\Z^\times}:=\left\{\la\in\La_\Z\,|\,\la_i\not=0\ \text{for all $i$}\,\right\}, \quad \La^{+}_{\Z^\times}:=\La^+\cap\La_{\Z^\times},
\end{equation*}
and let   a bijection $\sharp:\La_\hZ\rightarrow \La_{\Z^\times}$ be given by
\begin{align}\label{sharp:map}
\la=\sum_{i=1}^n\la_i\ep_i\ {\longmapsto} \ \la^\sharp:=\sum_{i=1}^n\left(\la_i+{\rm sgn}(\la_i)\hf\right)\ep_i.
\end{align}

\section{Bruhat orderings}

Let $\La^{\geqslant}=\{\,\la=\sum_{i=1}^n\la_i\varepsilon_i\in\La\,|\,\la_i\ge\la_{i+1} \text{ for all $i$}\,\}$ be the set of $\G_{\bar 0}$-dominant weights. We have $\La^+\subset\La^{\geqslant}$. Put $\La^\geqslant_{\epsilon+\Z}=\La^\geqslant\cap \La_{\epsilon+\Z}$ for $\epsilon=0,\hf$, and $\La^\geqslant_{\Z^\times}=\La^\geqslant\cap\La_{\Z^\times}$.
For $\la, \mu\in \La^\geqslant$, we define the {\em Bruhat ordering} $\succcurlyeq$ as in \cite[Lemma 2.1]{PS2}: $\la\succcurlyeq \mu$ if and only if there exists a sequence of elements $\mu=\nu_{(1)},\ldots,\nu_{(k)}=\la$ in $\La^{\geqslant}$ and roots $\beta_i\in\Phi^+$ such that $\nu_{(i)}+\beta_i=\nu_{(i+1)}$ with $(\nu_{(i)},\beta_i)=0$ for $1\leq i\leq k-1$.

Let us say that a subset $A$ of a partially ordered set $(S,\geq)$ is increasing if we have $s\in A$ for any $a\in A$ and $s\in S$ with $s\geq a$.

\begin{lem}\label{lem:partial order non-negativity}
The set $\La^\geqslant_{\Z^\times}$ is an increasing subset of $(\La^\geqslant,\succcurlyeq)$.
\end{lem}

\begin{proof}
Let $\la,\mu\in\La^\geqslant$ such that $\la\succcurlyeq\mu$. We will show that if $\mu\in \La^\geqslant_{\Z^\times}$, then $\la\in \La^\geqslant_{\Z^\times}$.
It is enough to consider the case $\la-\mu=\ep_i-\ep_j$ with $i<j$ and $(\mu,\ep_i-\ep_j)=0$. In this case we have $\mu_i>0$ and $\mu_j=-\mu_i<0$. Thus $\la_i=\mu_i+1>0$ and $\la_j=\mu_j-1<0$. As $\la_s=\mu_s$, for all $s\not=i,j$, the lemma follows.
\end{proof}

%\begin{cor}
%Let $\la,\mu\in\La^+_\Z$ such that $\la\succcurlyeq\mu$. Then $\nu\in \La^+_{\Z^\times}$ for any $\nu\in \La^+_\Z$ with $\la\succcurlyeq \nu\succcurlyeq \mu$.
%\end{cor}

\begin{lem}\label{lem:partial order}
For $\la, \mu\in \La^\geqslant_{\hf+\Z}$, we have $\la\succcurlyeq \mu$ if and only if $\la^\sharp\succcurlyeq \mu^\sharp$.
\end{lem}
\begin{proof}
Suppose that $\la\succcurlyeq \mu$.
Again it is enough to prove the case when $\la-\mu=\varepsilon_i-\varepsilon_j$ and $(\mu,\varepsilon_i-\varepsilon_j)=0$ for some $1\leq i<j\leq n$. We have $\la_i=\mu_i+1>\mu_i>0>\mu_j>\mu_j-1=\la_j$ with $\mu_i+\mu_j=0$.
Then it is clear that $\la^\sharp-\mu^\sharp=\varepsilon_i-\varepsilon_j$ and hence $\la^\sharp\succcurlyeq \mu^\sharp$.

Conversely, suppose that $\la^\sharp \succcurlyeq \mu^\sharp$. By Lemma \ref{lem:partial order non-negativity} we have $\eta\in \La^\geqslant_{\Z^\times}$ for any $\eta\in \La^\geqslant_\Z$ such that $\la^\sharp\succcurlyeq \eta\succcurlyeq \mu^\sharp$, that is, $\eta=\nu^\sharp$ for some (unique) $\nu\in \La^\geqslant_{\hf+\Z}$. Then it is clear that $\la\succcurlyeq \mu$ by the same argument as in the above paragraph.
\end{proof}

In \cite[Section 2.3]{Br2} a partial ordering $\succeq$ on $\La_\Z$ was defined based on the root lattice of the Lie algebra of type $\mf b_\infty$ (i.e., type $B$ of infinite rank), which we shall recall below. We will show that this ordering on $\La^{\geqslant}_\Z$ is equivalent to $\succcurlyeq$ in Proposition \ref{prop:Bruhat PS-B} below. So there will be no confusion to call both orderings Bruhat ordering in this paper.

Let $\texttt{P}$ be the free abelian group with orthonormal basis
$\{\,\delta_r\,\vert\, r\in\N\,\}$ with respect to a bilinear form
$(\cdot , \cdot)$. We define a partial ordering on $\texttt{P}$ by
declaring $\nu\ge\eta$, for $\nu,\eta\in \texttt{P}$, if $\nu-\eta$ is
a non-negative integral linear combination of the positive simple roots of type $\mf b_\infty$, that is,  $-\delta_1$ and
$\delta_r-\delta_{r+1}$ ($r \in \N$).

For $\la=\sum_{i=1}^n\la_i\varepsilon_i\in\La_\Z$, let
\begin{align*}
\text{wt}_s(\la)
 :=\sum_{s\le i\le n}\delta_{\la_i}\in \texttt{P},\quad
\text{wt}(\la):=\text{wt}_{1}(\la)\in \texttt{P},
\end{align*}
for $1\leq s\leq n$, where by definition we have $\delta_{-r}=-\delta_r$ for $r\in\N$ and $\delta_0=0$.
Then we define a partial order $\succeq$ on $\La_\Z$ as follows: $\la\succeq\mu$ if and only if
$\text{wt}(\la)=\text{wt}(\mu)$ and $\text{wt}_s(\la)\ge\text{wt}_s(\mu)$ for all $1\leq s\leq n$.

\begin{prop}\label{prop:Bruhat PS-B}
For $\la, \mu\in \La^\geqslant_\Z$, we have $\la\succcurlyeq \mu$ if and only if $\la\succeq \mu$.
\end{prop}

\begin{proof}
Let $\la=\sum_{i=1}^n\la_i\varepsilon_i$ and $\mu=\sum_{i=1}^n\mu_i\varepsilon_i$.
Suppose that $\la\succcurlyeq \mu$. If is enough to show the case when $\la-\mu=\ep_i-\ep_j$ with $i<j$ and $(\mu,\ep_i-\ep_j)=0$. Let $a=\la_i$.
Then
\begin{equation*}\label{eq:difference of b-weight}
{\rm wt}_s(\la)-{\rm wt}_s(\mu)=
\begin{cases}
0, & \text{if $j<s\le n$},\\
\delta_{a-1}-\delta_a, & \text{if $i<s\leq j$ and $a>1$},\\
-\delta_{1}, & \text{if $i<s\leq j$ and $a=1$},\\
0, & \text{if $1\le s\le i$},\\
\end{cases}
\end{equation*}
which implies that $\la\succeq\mu$.

Next, suppose that $\la\succeq \mu$. Choose a $p$ such that $\la_i, \mu_i\geq 0$ for $1\leq i\leq p$, and $\la_j, \mu_j\leq 0$ for $p+1\leq j\leq n$. Let us identify $\la$, $\mu$ with sequences
\begin{equation}\label{eq:antidominant map}
\begin{split}
(\,a_1,\ldots,a_p\,|\,b_1,\ldots,b_q\,),\ \
(\,a'_1,\ldots,a'_p\,|\,b'_1,\ldots,b'_q\,),\\
\end{split}
\end{equation}
respectively, where $a_i=\la_i$, $a'_i=\mu_i$ for $1\leq i\leq p$, and $b_j=-\la_{p+j}$, $b'_j=-\mu_{p+j}$ for $1\leq j\leq q:=n-p$. Let $K(\la,\mu)=A-A'+B-B'$, where $A=\sum_ia_i$, $A'=\sum_{i}a'_i$, $B=\sum_jb_j$, and $B'=\sum_jb'_j$.

We see first that $A\geq A'$ by \cite[Lemma 2.15]{Br2}. We claim next that $b_j\geq b'_j$ for all $j$. It is clear that $b_q\geq b'_q$ since ${\rm wt}_q(\la)\geq {\rm wt}_q(\mu)$. Suppose that there exists a $k$ such that $b_j\geq b'_j$ for $k<j\leq q$ but $0\leq b_k<b'_k$. Let $a=b_k$ and $a'=b'_k$. Then we have
\begin{equation}\label{eq:diff negative}
{\rm wt}_{p+k}(\la)-{\rm wt}_{p+k}(\mu)= {\rm wt}_{p+k+1}(\la)-{\rm wt}_{p+k+1}(\mu)  - (\delta_{a} - \delta_{a'}),
\end{equation}
which contradicts the fact that ${\rm wt}_{p+k}(\la)\geq {\rm wt}_{p+k}(\mu)$ since ${\rm wt}_{p+k+1}(\la){-} {\rm wt}_{p+k+1}(\mu)$ is a non-negative integral combination of $\delta_r-\delta_{r+1}$ for $r\geq a'$. This proves the claim. In particular, we have $B\geq B'$, and hence $K(\la,\mu)\geq 0$.

Suppose that $K(\la,\mu)=0$, equivalently, $A=A'$ and $B=B'$. Then we have $b_j=b'_j$ for all $1\leq j\leq q$ since $b_j\geq b'_j$ for $1\leq j\leq q$. Suppose that $a_i\neq a'_i$ for some $1\leq i\leq p$. Suppose that there exists a $k$ such that $a_i\leq a'_i$ for $k<i\leq n$ and $a_k>a'_k$. Then
\begin{equation}\label{eq:diff negative-2}
{\rm wt}_{k}(\la)-{\rm wt}_{k}(\mu)= {\rm wt}_{k+1}(\la)-{\rm wt}_{k+1}(\mu)  - (\delta_{a'} - \delta_{a}),
\end{equation}
where $a=a_k$ and $a'=a'_k$. This yields the same contradiction as in \eqref{eq:diff negative}. So we must have $a'_i\geq a_i$ and hence $a'_i =a_i$ for $1\leq i\leq n$ since $A=A'$.

Now, we use induction on $K(\la,\mu)$ and $n$ to show that $\la\succcurlyeq\mu$. If $n=1$, then it is clear. Also, if $K(\la,\mu)=0$, then we have $\la=\mu$ by the argument in the previous paragraph. So we assume that $K(\la,\mu)>0$ and $n\geq 2$.

If $q=0$, then it is clear that $\la=\mu$.
If $q\geq 1$ and $b_q=b'_q$, then we may apply the induction hypothesis to the weights
\begin{equation*}
\begin{split}
(\,a_1,\ldots,a_p\,|\,b_1,\ldots,b_{q-1}\,),\ \
(\,a'_1,\ldots,a'_p\,|\,b'_1,\ldots,b'_{q-1}\,),\\
\end{split}
\end{equation*}
for $\mf q(n-1)$ to conclude that $\la \succcurlyeq \mu$.

Suppose that $q\geq 1$ and $b_q> b'_q$. Then $b'_j\neq b_q$ for all $1\leq j\leq q$, and hence $a_i=b_q$ for some $1\leq i\leq p$. Put $a=a_i=b_q$. Consider
\begin{equation*}
\gamma:=\la-\varepsilon_u+\varepsilon_v = (\,a_1,\ldots, a_u-1,\ldots ,a_p\,|\,b_1,\ldots,b_{v}-1,\ldots,b_q\,),
\end{equation*}
where $u$ and $v$ are such that $a_u=a>a_{u+1}$ and $b_{v+1}<b_v=a$.
It is clear that $\la\succcurlyeq\gamma$.
On the other hand, we have
\begin{equation*}
{\rm wt}_s(\gamma)-{\rm wt}_s(\mu)=
\begin{cases}
{\rm wt}_s(\la)-{\rm wt}_s(\mu), & \text{if $v<s\leq n$},\\
{\rm wt}_s(\la)-{\rm wt}_s(\mu) - \delta_{a-1}+\delta_a, & \text{if $u<s\leq v$ and $a>1$},\\
{\rm wt}_s(\la)-{\rm wt}_s(\mu) +\delta_{1}, & \text{if $u<s\leq v$ and $a=1$},\\
{\rm wt}_s(\la)-{\rm wt}_s(\mu), & \text{if $1\le s\le u$}.
\end{cases}
\end{equation*}
Assume that when $u< s \leq v$ and $a> 1$, we have ${\rm wt}_s(\la)-{\rm wt}_s(\mu)=c_{s,0}(-\delta_1)+\sum_{1\leq r\leq a-2}c_{s,r}(\delta_r-\delta_{r+1}) + c_{s,a-1}(\delta_{a-1}-\delta_a)$ for some $c_{s,r}\in \Z_{+}$ ($0\leq r\leq a-1$). Since $b'_j<b_q=a$ for all $j$, we have $c_{s,a-1}\geq q-v+1$, and so
\begin{equation*}
\begin{split}
&{\rm wt}_s(\la)-{\rm wt}_s(\mu) - \delta_{a-1} +\delta_a \\
&=c_{s,0}(-\delta_1)+\sum_{1\leq r\leq a-2}c_{s,r}(\delta_r-\delta_{r+1}) + (c_{s,a-1}-1)(\delta_{a-1}-\delta_a),
\end{split}
\end{equation*}
which implies that $\gamma\succeq\mu$. Similarly, when when $u< s \leq v$ and $a=1$, we have ${\rm wt}_s(\la)-{\rm wt}_s(\mu)=c_{s,0}(-\delta_1)$ for some $c_{s,0}\geq 1$, and hence ${\rm wt}_s(\la)-{\rm wt}_s(\mu) + \delta_{1}=(c_{s,0}-1)(-\delta_1)$, which also implies $\gamma\succeq \mu$.
Since $K(\gamma,\mu)<K(\la,\mu)$, we have $\gamma\succcurlyeq \mu$ by induction hypothesis. Therefore, $\la\succcurlyeq\mu$. The proof completes.
\end{proof}

\begin{cor}\label{cor:succcurlyeq = succeq}
For $\la, \mu\in \La^+_{\hf+\Z}$, we have $\la\succcurlyeq \mu$ if and only if $\la^\sharp\succeq \mu^\sharp$.
\end{cor}

The Bruhat order $\succeq $ on $\La_\Z$ is also related with the Bruhat order on the integral weight lattice for the general linear Lie superalgebra in \cite[Section 2-b]{Br1} as follows.
For $0\leq p\leq n$ and $q=n-p$, let $\Z^{p|q}$ be the set of integer-valued functions on $J(p|q)=\{\,-p<\cdots<-1<1<\cdots<q\,\}$. We shall also identify an element $f\in\Z^{p|q}$ with the
sequence of integers $$f=(f(-p),\ldots,f(-1)\,|\,f(1),\ldots, f(q)),$$ so that $\Z^{p|q}$ can be understood as the integral weight lattice for  $\mf{gl}(p|q)$.
%Let $\Z^{p|q}_+$ be the set of integral dominant weights for $\mf{gl}(p|q)$, that is, $f\in \Z^{p|q}$ such that $f(-p)<\cdots<f(-1)$ and $f(1)>\cdots>f(q)$.
Let $P$ be the free abelian group with orthonormal basis
$\{\,\gamma_r\,\vert\, r\in\Z\,\}$ with respect to a bilinear form
$(\cdot , \cdot)$. We define a partial order on $P$ by
declaring $\nu\ge\eta$, for $\nu,\eta\in P$, if $\nu-\eta$ is
a non-negative integral linear combination of the positive simple roots of $\mf{gl}_\infty$, that is, $\gamma_r-\gamma_{r+1}$ ($r \in \Z$).
For $f\in \Z^{p|q}$, let
\begin{align*}
\text{wt}_s(f)
 :=\sum_{s\le i\le n}{\rm sgn}(i)\gamma_{f(i)}\in P,\quad
\text{wt}(f):=\text{wt}_{-p}(f)\in P,
\end{align*}
for $s\in J(p|q)$.
Then for $f, g\in \Z^{p|q}$ we define $f\succeq_a g$ if and only if
$\text{wt}(f)=\text{wt}(g)$ and $\text{wt}_s(f)\ge\text{wt}_s(g)$ for all $s\in J(p|q)$.

Now, let
\begin{equation}\label{def:p:q}
\La^\geqslant_{\Z^\times}(p):=\left\{\,\la=\sum_{i=1}^n\la_i\varepsilon_i\in \La^\geqslant_{\Z^\times}\,\Bigg\vert\,\la_p>0>\la_{p+1}\,\right\}.
\end{equation}
We have $\La^\geqslant_{\Z^\times}=\bigsqcup_{p=0}^n \La^\geqslant_{\Z^\times}(p)$.
For $\la\in \La^\geqslant_{\Z^\times}(p)$, define
\begin{equation}\label{eq:dominant map}
\la^\flat := (-\la_1,\ldots,-\la_p\,|\,\la_{p+1},\ldots,\la_n)\in \Z^{p|q},
\end{equation}
(cf.~\eqref{eq:antidominant map}).
%The image of $\La^\geqslant_{\Z^\times}(p)$ under the map $\la\mapsto\la^\flat$ is the set of $f\in \Z^{p|q}_+$ with negative entries, which we denote by $(\Z^{p|q}_+)^{-}$.
%Let $\succeq_a$ denote the partial order on $\La^+_{\Z^\times}$ induced from $\succeq_a$ on $\Z^{p|q}$ for $0\leq p\leq n$ under the map \eqref{eq:dominant map}.
Let $\Z^{p|q}_{\geqslant}$ be the set of $f\in \Z^{p|q}$ such that $f(-p)\leq\cdots\leq f(-1)$ and $f(1)\geq\cdots\geq f(q)$.

\begin{lem}\label{lem:increasing property} The map $\flat$ is an injection and the
image of $\La^\geqslant_{\Z^\times}(p)$ under $\flat$ is an increasing subset of $(\Z^{p|q}_{\geqslant},\succeq_a)$ for $0\leq p\leq n$.
\end{lem}
\begin{proof}
It follows directly from \cite[Lemma 2.5]{Br1}.
\end{proof}

\begin{prop}\label{prop:succeq and succeq'}
For $\la, \mu\in \La^\geqslant_{\Z^\times}$, we have $\la\succeq\mu$ if and only if $\la^\flat \succeq_a \mu^\flat$.
\end{prop}
\begin{proof}
Suppose that $\la\succeq\mu$. By Lemma \ref{lem:partial order non-negativity} and Proposition \ref{prop:Bruhat PS-B}, there exists a sequence of elements $\mu=\nu_{(1)},\ldots,\nu_{(k)}=\la$ in $\La^{\geqslant}_{\Z^\times}$ and roots $\beta_i\in\Phi^+$ such that $\nu_{(i)}+\beta_i=\nu_{(i+1)}$ with $(\nu_{(i)},\beta_i)=0$ for $1\leq i\leq k-1$. Then we have $\la, \mu\in \La^\geqslant_{\Z^\times}(p)$ for some $p$, and
$\nu_{(i+1)}^\flat=\nu_{(i)}^\flat - d_{-u_i} - d_{v_i}$ for some $1\leq u_i\leq p$ and $1\leq v_i\leq q$, where $d_s\in \Z^{p|q}$ is given by $d_s(t)=\delta_{st}$ for $s,t\in J(p|q)$. (This implies $\nu_{(i+1)}^\flat \downarrow \nu_{(i)}^\flat$ following the notation in \cite[Section 2-b]{Br1}.) Thus we have $\la^\flat \succeq_a \mu^\flat$ by \cite[Lemma 2.5]{Br1}.

Conversely, suppose that $\la^\flat \succeq_a \mu^\flat$. Again by \cite[Lemma 2.5]{Br1} there exists a sequence of elements  $\mu^\flat =f_{(1)},\ldots, f_{(k)}=\la^\flat$  in $\Z^{p|q}$ such that $f_{(i+1)}\downarrow f_{(i)}$ for $1\leq i\leq k-1$. Since $\la^\flat, \mu^\flat\in \Z^{p|q}_{\geqslant}$, we can find a sequence of elements  $\mu^\flat = g_{(1)},\ldots, g_{(l)}=\la^\flat$ in $\Z^{p|q}_{\geqslant}$ such that
$g_{(i+1)}=g_{(i)} - d_{-u_i} - d_{v_i}$ for some $1\leq u_i\leq p$ and $1\leq v_i\leq q$. Moreover by Lemma \ref{lem:increasing property}, $g_{(i)}=\nu_{(i)}^\flat$ for some $\nu_{(i)}\in \La^{\geqslant}_{\Z^\times}$. This implies that $\nu_{(i)}+\beta_i=\nu_{(i+1)}$ with $(\nu_{(i)},\beta_i)=0$ for some $\beta_i\in\Phi^+$. Therefore, $\la\succeq \mu$.
\end{proof}

\begin{cor}
For $\la, \mu\in \La^+_{\hf+\Z}$, we have $\la\succcurlyeq \mu$ if and only if $\la^\natural\succeq_a \mu^\natural$, where $\la^\natural=(\la^\sharp)^\flat$.
\end{cor}

\section{Dual Zuckerman functors}

Let $\mf b\subseteq\mf p \subseteq \G$ be parabolic subalgebras with Levi subalgebra $\mf h\subseteq \mf l\subseteq \G$, respectively.
Let $\mc{HC}(\G,\mf{l}_\even)$ be the category of $\G$-modules that are direct sums of finite-dimensional simple $\mf{l}_\even$-modules, and let $\mc{HC}(\G,\mf{\G}_{\bar 0})$ be defined similarly.
Let $\mc L^{\G,\mf l}$ be the {\em dual Zuckerman functor} from $\mc{HC}(\G,\mf{l}_\even)$ to $\mc{HC}(\G,\G_\even)$ as in \cite[Section 4]{San}, which is a right exact functor. For $i\ge 0$, we denote by $\mc L^{\G,\mf l}_i$ its $i$th derived functor. For $M\in \mc{HC}(\G,\mf{l}_\even)$, we let
\begin{align*}
\mc E^{\G,\mf l}(M):=\sum_{i\ge 0}(-1)^i\mc L^{\G,\mf l}_i(M)
\end{align*}
be the {\em Euler characteristic of $M$}, which is a virtual $\G$-module.

By the same arguments as in \cite[Sections 4 and 5]{San}, we can check the following.

\begin{prop}\label{prop:same:central}
Let $\la\in\h_\even^*$ such that the irreducible $\mf l$-module $L(\mf{l},\la)$ with highest weight $\la$ is finite dimensional so that $M:={\rm Ind}_{\mf p}^{\G}L(\mf{l},\la)\in\mc{HC}(\G,\mf{l}_\even)$.
\begin{itemize}
\item[(1)] The $\G$-module $\mc L^{\G,\mf l}_0\left(M\right)$ is the maximal finite-dimensional quotient of $M$.

\item[(2)] The $\G$-module $\mc L^{\G,\mf l}_i\left(M\right)$ is finite dimensional for all $i\geq 0$, and $\mc L^{\G,\mf l}_i\left(M\right)=0$ for $i\gg 0$.

\item[(3)] Let $I$ be the annihilator of $M$ in $U(\G)$. Then $I$ annihilates every $\mc{L}^{\G,\mf l}_i\left(M\right)$ for all $i\ge 0$. In particular, all the $\G$-modules $\mc L^{\G,\mf l}_i\left(M\right)$ have the same central character.

\item[(4)] The character of the Euler characteristic of $M$ is given by
\begin{align*}
{\rm ch}\mc E^{\G,\mf l}(M)=D^{-1} \sum_{w\in W}(-1)^{\ell(w)}w\left( \frac{{\rm ch}L(\mf{l},\la)}{\prod_{\alpha\in\Phi^+(\mf{l_{\bar 1}})}(1+e^{-\alpha})} \right),
\end{align*}
where $\Phi^+(\mf{l_{\bar 1}})$ denotes the set of positive roots of $\mf l_\odd$ and
\begin{equation*}
%\begin{split}
D_{\bar 0}:=\prod_{\alpha\in\Phi^+_{\bar 0}}(e^{\alpha/2}-e^{-\alpha/2}),\quad
D_{\bar 1}:=\prod_{\alpha\in\Phi^+_{\bar 1}}(e^{\alpha/2}+e^{-\alpha/2}),\quad
D :=\frac{D_{\bar 0}}{D_{\bar 1}}.
%\end{split}
\end{equation*}

\end{itemize}
\end{prop}

For an exact sequence in $\mc{HC}(\G,\mf{l}_\even)$
\begin{align*}
0\longrightarrow M\longrightarrow E\longrightarrow N\longrightarrow 0,
\end{align*}
we have the following identity of virtual modules:
\begin{align}\label{Euler:additive}
\mc E^{\G,\mf l}\left(E\right)=\mc E^{\G,\mf l}\left(M\right)+\mc E^{\G,\mf l}\left(N\right).
\end{align}
Thus, Proposition \ref{prop:same:central}(4) remains valid with $L(\mf l,\la)$ replaced by a finite-dimensional $\mf l$-module.

%Let $\mf b\subseteq\mf p\subseteq \mf q\subseteq \G$ be parabolic subalgebras with Levi subalgebras $\mf h\subseteq\mf l\subseteq \mf k\subseteq\G$, respectively.

\begin{rem}
Let $\mf b\subseteq\mf p\subseteq \mf q\subseteq \G$ be parabolic subalgebras with Levi subalgebras $\mf h\subseteq\mf l\subseteq \mf k\subseteq\G$, respectively. Let $\mf u^{\G,\mf l}$ and $\mf u^{\G,\mf k}$ be the nilradicals of $\mf p$ and $\mf q$ such that $\mf p=\mf l+\mf u^{\G,\mf l}$ and $\mf q=\mf k+\mf u^{\G,\mf k}$, respectively. Let $\mf u_-^{\G,\mf l}$ and $\mf u_-^{\G,\mf k}$ denote the respective opposite nilradicals. Suppose that $M\in \mc{HC}(\mf{l},\mf{l}_\even)$ which we extend to a $\mf p$-module by $\mf u^{\G,\mf l}M=0$.  Consider the $\mf q$-module ${\rm Ind}_{\mf p}^{\mf q}M = U(\mf k\cap \mf u_-^{\G,\mf l})\otimes M$. Since $[\mf k,\mf u^{\G,\mf k}]\subseteq \mf u^{\G,\mf k}\subseteq \mf u^{\G,\mf l}$, we have $\mf u^{\G,\mf k}{\rm Ind}_{\mf p}^{\mf q}M=0$. We conclude that ${\rm Ind}_{\mf p}^{\mf q}M$ is a $\mf k$-module that is trivially extended to a $\mf q$-module.
\end{rem}

The following is the analogue of \cite[Theorem 1]{PS2} and \cite[Theorem 1]{GS} using the language of dual Zuckerman functor.

\begin{lem}\label{lem:nested:Euler} Let $\mf b\subseteq\mf p\subseteq \mf q\subseteq \G$ be parabolic subalgebras with Levi subalgebras $\mf h\subseteq\mf l\subseteq \mf k\subseteq\G$, respectively. For $M\in \mc{HC}(\mf{l},\mf{l}_\even)$, we have
\begin{align*}
\mc E^{\G,\mf l}\left({\rm Ind}^{\G}_{\mf p}M\right)
=
\mc E^{\G,\mf k}\left({\rm Ind}^{\G}_{\mf q} \mc E^{\mf k,\mf l}\left({\rm Ind}^{\mf q}_{\mf p}M\right)\right),
\end{align*}
where $M$ is regarded as a $\mf p$-module by letting the nilradical act trivially.
\end{lem}

\begin{proof}
%By \eqref{Euler:additive}, it is enough to show that $M=L(\mf l,\la)$.
First we have from the transitivity of dual Zuckerman functors (cf.~\cite[Proposition 2.19]{KV}) that for $M\in\mc{HC}(\mf{l},\mf{l}_\even)$
\begin{equation}\label{eq:transitivity}
\mc E^{\G_\even,\mf l_\even}\left({\rm Ind}^{\G_\even}_{\mf p_\even}M\right)
=
\mc E^{\G_\even,\mf k_\even}\left({\rm Ind}^{\G_\even}_{\mf q_\even} \mc E^{\mf k_\even,\mf l_\even}\left({\rm Ind}^{\mf q_\even}_{\mf p_\even}M\right)\right),
\end{equation}
where we regard $M$ as an $\mf l_\even$-module. Also from Mackey isomorphism (cf.~\cite[Chapter II.5]{KV}, we can deduce that for a finite-dimensional $\G_\even$-module $N$
\begin{equation}\label{eq:Mackey}
N\otimes \mc E^{\G_\even,\mf l_\even}(M) \cong \mc E^{\G_\even,\mf l_\even}(N\otimes M),
\end{equation}

Let $\mf u^{\G,\mf k}_{-}$, $\mf u^{\mf k,\mf l}_{-}$, and $\mf u^{\mf g,\mf l}_{-}$ be the opposite nilradicals such that $\G=\mf q \oplus \mf u^{\G,\mf k}_{-}$, $\mf q=\mf p \oplus \mf u^{\mf k,\mf l}_{-}$, and $\G=\mf p \oplus \mf u^{\G,\mf l}_{-}$, respectively.
Note that
\begin{equation}\label{eq:Res}
\begin{split}
&{\rm Res}^\G_{\G_\even}\mc E^{\G,\mf l}(U)=\mc E^{\G_\even,\mf l_\even}({\rm Res}^\G_{\G_\even}U),\\
&{\rm Res}^\G_{\G_\even}{\rm Ind}^{\G}_{\mf p}V \cong {\rm Ind}^{\G_\even}_{\mf p_\even}\left(\La((\mf u^{\G,\mf l}_{-})_\odd)\otimes {\rm Res}^{\mf l}_{\mf l_\even}V \right),
\end{split}
\end{equation}
for $U\in \mc{HC}(\G,\mf l_\even)$ and $V\in \mc{HC}({\mf l},{\mf l}_\even)$. For simplicity, let us omit the notation of ${\rm Res}$ if there is no confusion.

Now for $M\in \mc{HC}({\mf l},{\mf l}_\even)$, we have as $\G_\even$-modules
{\allowdisplaybreaks
\begin{align*}
&\mc E^{\G,\mf k}\left({\rm Ind}^{\G}_{\mf q} \mc E^{\mf k,\mf l}\left({\rm Ind}^{\mf q}_{\mf p}M\right)\right)\\
&\cong
\mc E^{\G_\even,\mf k_\even}\left({\rm Ind}^{\G_\even}_{\mf q_\even}\left(\La((\mf u^{\G,\mf k}_{-})_\odd)\otimes \mc E^{\mf k_\even,\mf l_\even}\left({\rm Ind}^{\mf q_\even}_{\mf p_\even}\left(\La((\mf u^{\mf k,\mf l}_{-})_\odd)\otimes M\right)\right)\right)\right) \quad\quad\quad \text{by \eqref{eq:Res}}\\
&\cong
\mc E^{\G_\even,\mf k_\even}\left({\rm Ind}^{\G_\even}_{\mf q_\even}\left(\mc E^{\mf k_\even,\mf l_\even}\left(\La((\mf u^{\G,\mf k}_{-})_\odd)\otimes{\rm Ind}^{\mf q_\even}_{\mf p_\even}\left(\La((\mf u^{\mf k,\mf l}_{-})_\odd)\otimes M\right)\right)\right)\right)
\quad\quad\quad \text{by \eqref{eq:Mackey}}\\
&\cong
\mc E^{\G_\even,\mf k_\even}\left({\rm Ind}^{\G_\even}_{\mf q_\even}\left(\mc E^{\mf k_\even,\mf l_\even}\left({\rm Ind}^{\mf q_\even}_{\mf p_\even}\left(\La((\mf u^{\G,\mf k}_{-})_\odd)\otimes\La((\mf u^{\mf k,\mf l}_{-})_\odd)\otimes M\right)\right)\right)\right)\\
&\cong
\mc E^{\G_\even,\mf k_\even}\left({\rm Ind}^{\G_\even}_{\mf q_\even}\left(\mc E^{\mf k_\even,\mf l_\even}\left({\rm Ind}^{\mf q_\even}_{\mf p_\even}\left(\La((\mf u^{\G,\mf l}_{-})_\odd) \otimes M\right)\right)\right)\right)\\
&\cong
\mc E^{\G_\even,\mf l_\even}\left({\rm Ind}^{\G_\even}_{\mf p_\even} \left(\La((\mf u^{\G,\mf l}_{-})_\odd) \otimes M \right)\right)
\hskip 5.9cm \text{by \eqref{eq:transitivity}}\\
&\cong
\mc E^{\G,\mf l}\left({\rm Ind}^{\G}_{\mf q}M\right).
\end{align*}}
The proof is complete.
\end{proof}

For $\la\in \La^+$, we may understand $\la$ as a weight for $\mf l$ or $\mf k$. We regard  the irreducible $\mf l$-module $L(\mf l,\la)$ with highest weight $\la$ as an irreducible $\mf p$-module by letting the nilradical act trivially. Similarly, we have the irreducible $\mf k$-module $L(\mf k,\la)$ which can be viewed as an irreducible $\mf q$-module.
For $\la, \mu\in \La^+$, we let
\begin{align*}
m^{\mf q,\mf p}_i(\la,\mu):=\left[\mc L^{\mf k,\mf l}_i\left({\rm Ind}_{\mf p}^{\mf q}L(\mf l,\la)\right):L(\mf k,\mu)\right],
\end{align*}
the multiplicity of the irreducible $\mf k$-module $L(\mf k,\mu)$ in $\mc L^{\mf k,\mf l}_i\left({\rm Ind}_{\mf p}^{\mf q}L(\mf l,\la)\right)$.
%{\color{red} It is well-known that $m^{\mf q,\mf p}_i(\la,\mu)\not=0$ only if $\la\succcurlyeq\mu$ with respect to $\mf k$ (or rather with respect to the simple roots of $\mf k$), and also $m^{\mf q,\mf p}_i(\la,\la)=\delta_{i0}$ (see, e.g., the proof of \cite[Lemma 2.8]{PS2}).(Remove???)}

Let $\la\in \La^+$ be given. Put
\begin{align}\label{llambda}
\mf l(\la)=\mf h\oplus\bigoplus_{\alpha\in \Phi^+(\la)}(\G_\alpha\oplus \G_{-\alpha}),\quad \mf p(\la)=\mf l(\la) +\mf b,
\end{align}
where $\Phi^+(\la)=\{\,\varepsilon_i-\varepsilon_j\,|\,\, \la_i=\la_j\ (i< j)\,\}$ and $\G_\alpha$ is the root space of $\G$ corresponding to $\alpha$. Note that $\mf l(\la)=\mf h$ for $\la\in \La^+_{\hf+\Z}$.

Let $E(\la):=\mc E^{\G,\mf l(\la)}\left({\rm Ind}_{\mf p(\la)}^\G W_\la\right)$ and here $W_\la$ is regarded as an irreducible $\mf l(\la)$-module. We have by Proposition~\ref{prop:same:central}(4):
\begin{equation}\label{eq:Euler character}
{\rm ch}E(\la)=2^{\lceil \ell(\la)/2 \rceil}D^{-1}\sum_{w\in W}(-1)^{\ell(w)}w\left(\frac{e^\la}{\prod_{\beta\in \Phi^+(\la)}(1+e^{-\beta})} \right).
\end{equation}
It is well-known that we have the following two bases in $K(\mc O^{\texttt{fin}}_{\epsilon+\Z})$:
\begin{equation*}
\left\{\,[E(\la)]\,\big\vert \,\la\in\La^+_{\epsilon+\Z}\,\right\},\quad \left\{\,[L(\la)]\,\big\vert \,\la\in\La^+_{\epsilon+\Z}\,\right\}.
\end{equation*}
We thus have
\begin{align}\label{E:in:L}
[E(\la)]=\sum_{\mu\in\La^+_{\epsilon+\Z}}a_{\la\mu}[L(\mu)],
\end{align}
where $a_{\la\mu}=\sum_{i\ge 0}(-1)^im_i^{\G,\mf p(\la)}(\lambda,\mu)$.%, and $a_{\la\mu}\not=0$ only if $\la\succcurlyeq\mu$ by \cite[Corollary 2.3]{PS2}.

Let
\begin{align}\label{seq:parabolic}
\mf p(\la)= \mf p_0\subseteq \mf p_1\subseteq\cdots\subseteq\mf p_{k-1}\subseteq\mf p_k=\G
\end{align}
be a sequence of parabolic subalgebras with respective sequence of Levi subalgebras
\begin{align*}
\mf l(\la)=\mf l_0\subseteq \mf l_1\subseteq\cdots\subseteq\mf l_{k-1}\subseteq\mf l_k=\G.
\end{align*}
We have
\begin{align*}
{\rm Ind}_{\mf p(\la)}^\G W_\la = {\rm Ind}_{\mf p_{k-1}}^{\mf p_k}\cdots{\rm Ind}_{\mf p_1}^{\mf p_2}{\rm Ind}_{\mf p_0}^{\mf p_1} W_\la.
\end{align*}
If we put $\mc E^{s,s-1}=\mc E^{\mf l_s, \mf l_{s-1}}$, ${\rm Ind}_{s-1}^{s}={\rm Ind}_{\mf p_{s-1}}^{\mf p_s}$, and $m_i^{s,s-1}(\mu,\nu)=m_{i}^{\mf p_s,\mf p_{s-1}}(\mu,\nu)$ for $1\leq s\leq k$, $i\geq 0$, and $\mu,\nu\in \La^+$, then we have by Lemma~\ref{lem:nested:Euler}
{\allowdisplaybreaks
\begin{align*}
& E(\la)\\
& = \mc E^{k,0}\left({\rm Ind}_{{k-1}}^{k}\cdots{\rm Ind}_{1}^{2}{\rm Ind}_{0}^{1} W_\la\right)\\
&=\mc E^{k,k-1}\left({\rm Ind}_{{k-1}}^{k}\mc E^{k-1,k-2}\left(\cdots\mc E^{2,1}\left({\rm Ind}_{1}^{2}\mc E^{1,0}\left({\rm Ind}_{0}^{1} W_\la\right)\right)\cdots\right)\right)\\
&=\mc E^{k,k-1}\left( \cdots \mc E^{2,1}\left({\rm Ind}_{1}^{2}\sum_{i_1\ge 0, \  \la^{(1)}}(-1)^{i_1} m^{1,0}_{i_1}(\la,\la^{(1)})L(\mf l_1,\la^{(1)})\right)\cdots\right)\\
&=\mc E^{k,k-1}\left(\cdots\mc E^{3,2}\left({\rm Ind}_{2}^{3}\sum_{\substack{i_1,i_2\ge 0 \\  \la^{(1)}, \la^{(2)}}}(-1)^{i_1+i_2} m^{1,0}_{i_1}(\la,\la^{(1)}) m^{2,1}_{i_2}(\la^{(1)},\la^{(2)})L(\mf l_2,\la^{(2)})\right)\cdots\right)\\
&= \cdots \\
&=\sum_{\substack{i_1,\ldots,i_k\ge 0 \\ \la^{(1)},\ldots,\la^{(k)}}}(-1)^{i_1+\cdots+i_k} m^{1,0}_{i_1}(\la,\la^{(1)}) m^{2,1}_{i_2}(\la^{(1)},\la^{(2)})\cdots m^{k,k-1}_{i_k}(\la^{(k-1)},\la^{(k)})L(\la^{(k)}).
\end{align*}}
Comparing with \eqref{E:in:L} we get the following formula for $a_{\la\mu}$:
\begin{align}\label{formula:a}
\begin{split}
a_{\la\mu}&=\sum_{\substack{i_1,\ldots,i_k\ge 0 \\ \la^{(1)},\ldots,\la^{(k-1)}}}(-1)^{i_1+\cdots+i_k}
\prod_{s=1}^k m^{s,s-1}_{i_s}(\la^{(s-1)},\la^{(s)}),
\end{split}
\end{align}
where $\la^{(0)}=\la$ and $\la^{(k)}=\mu$. Note that for different $\mu$'s, we can choose in principle different sequences of parabolic subalgebras. This way we obtain a method to compute the coefficients $a_{\la\mu}$, if we can find a sequence of parabolic subalgebras as in \eqref{seq:parabolic} for which we can compute the multiplicity $m^{s,s-1}_i(\nu,\eta)$, for all relevant $\nu,\eta$. This is the essence of the algorithm in \cite{PS2}, which we shall use in the next section.

\section{Penkov-Serganova algorithm}

Below is an analogue of the ``typical lemma'' (\cite[Lemma 5]{GS}) for the queer Lie superalgebra. Although our proof below is algebraic in nature, it is inspired by loc.~cit.  We mention that this lemma already appeared in \cite[Theorem 2]{PS2}.

\begin{prop}\label{prop:typical1}
Let $\la=\sum_{i=1}^n\la_i\varepsilon_i \in \La^+$ be given. Suppose that $\mf l$ is a Levi subalgebra that contains every positive root atypical to $\la$.
Then we have
\begin{align*}
\mc L^{\G,\mf l}_i\left({\rm Ind}_{\mf p}^\G L(\mf l,\la)\right)=
\begin{cases}
L(\la),&\text{ if }i=0,\\
0,&\text{ if } i>0.
\end{cases}
\end{align*}
In particular, we have ${\rm ch}\mc{E}^{\G,\mf l}{\left({\rm Ind}_{\mf p}^\G L(\mf l,\la)\right)}={\rm ch}{L(\la)}$.
\end{prop}

\begin{proof}
First, we take a coweight $h_{\mf l}=\sum_{i=1}^nc_i h_i\in \h_\even$ satisfying
\begin{equation*}
\begin{cases}
c_i-c_{i+1}=0, & \text{if $\varepsilon_i -\varepsilon_{i+1}\in \Phi^+({\mf l}_\even)$},\\
c_i-c_{i+1}=1, & \text{if $\varepsilon_i -\varepsilon_{i+1}\not\in \Phi^+({\mf l}_\even)$},
\end{cases}
\end{equation*}
for $1\leq i\leq n-1$.
Then $\Phi({\mf l})=\{\,\alpha\in \Phi\,|\,\langle \alpha,h_{\mf l} \rangle=0\,\}$ and $\Phi({\mf p})=\{\,\alpha\in \Phi \,|\,\langle \alpha,h_{\mf l}\rangle\geq 0\,\}$, where $\Phi$, $\Phi({\mf l})$, and $\Phi({\mf p})$ denote the set of roots of $\G$, ${\mf l}$, and ${\mf p}$, respectively. Denote by $W_{\mf l}$ the Weyl group of ${\mf l}$.

Let $\mu\in\La^+$ and suppose that $L(\mu)$ occurs in $\mc L^{\G,\mf l}_i\left({\rm Ind}_{\mf p}^\G L(\mf l,\la)\right)$ for some $i\geq 0$. Now, observe that we have $L(\mf l,\la)=U(\mf{l}_\odd\cap\mf{n}_-)U(\mf{l}_\even\cap\mf{n}_-)W_\la=U(\mf{l}_\odd\cap\mf{n}_-)E$, where we recall that $\mf{n}_-$ is the opposite nilradical of $\mf{b}$, and $E$, as an $\mf{l}_\even$-module, is a direct sum of copies of the irreducible $\mf{l}_\even$-module of highest weight $\la$.
Thus, by \cite[Lemma 1.3.3]{Ger}, we have
\begin{equation}\label{eq:aux-1}
\mu=w\left(\la -\sum_{\alpha \in J}\alpha + \rho_\even\right) - \rho_\even,
\end{equation}
for some $w\in W$ with $\ell(w)=i$ and $J\subseteq\Phi^+_\odd$.
Since $\rho_\even=\rho_\odd$ and in the expression $-\sum_{\alpha \in J}\alpha + \rho_\odd$ every $\alpha\in\Phi^+_\odd$ appears exactly once with coefficient $\hf$ or $-\hf$, we have $w\left(-\sum_{\alpha \in J}\alpha + \rho_\odd\right) - \rho_\odd=-\sum_{\alpha\in K}\alpha$, for some $K\subseteq\Phi^+_\odd$, and hence
\begin{equation}\label{eq:aux-2}
\mu=w(\la) -\sum_{\alpha\in K}\alpha.
\end{equation}
Note that
$\la-w(\la)=\sum_{\alpha\in \Phi^+_\even}k_\alpha \alpha$ for some $k_\alpha\in \Z_+$ since $\la$ is a $\G_\even$-dominant weight. So we have
\begin{equation}\label{eq:aux-3}
\langle \mu,h_{\mf l}\rangle \leq\langle w(\la),h_{\mf l} \rangle \leq \langle \la,h_{\mf l} \rangle.
\end{equation}

Suppose that the degree of atypicality of $\la$ is $r$, and let $S_\la$ be a set of mutually orthogonal $r$ even positive roots atypical to $\la$. Since $\chi_\mu=\chi_\la$ by Proposition \ref{prop:same:central}(3), we have by the linkage principle for $\G$ (\cite{Sv}, see also \cite[Theorem 2.48]{CW})
$$v(\mu) = \la + \sum_{\alpha\in S_\la}c_\alpha\alpha,$$
for some $c_\alpha$ and $v\in W$.
Since $\langle v(\mu),h_{\mf l} \rangle \leq \langle \mu,h_{\mf l} \rangle$ as in \eqref{eq:aux-3}, we have
\begin{equation}\label{eq:aux-4}
\langle \mu,h_{\mf l} \rangle \geq \langle v(\mu),h_{\mf l} \rangle = \langle \la,h_{\mf l}\rangle + \sum_{\alpha\in S_\la}c_\alpha\langle \alpha,h_{\mf l}\rangle = \langle \la,h_{\mf l}\rangle.
\end{equation}
So we have
$\langle \mu,h_{\mf l}\rangle  =\langle w(\la),h_{\mf l} \rangle = \langle \la,h_{\mf l} \rangle$
by \eqref{eq:aux-3} and \eqref{eq:aux-4}, and $K\subset \Phi^+({\mf l})$.

Since $\langle w(\la),h_{\mf l} \rangle = \langle \la,h_{\mf l} \rangle$ and $\langle \la,\beta\rangle>0$ for $\beta\not\in \Phi^+({\mf l}_\even)$, we conclude that $w\in W_{\mf l}$, and then $J\subseteq \Phi_\odd^+(\mf l)$.
By construction of $\mu$ (cf.~\cite[Lemma 1.3.3]{Ger}), we see that $\la -\sum_{\alpha \in J}\alpha$ is a weight of $L({\mf l},\la)$ and ${\mf l}_\even$-dominant. Hence we have $w=id$, and in particular, $\mc{L}^{\G,\mf l}_i\left({\rm Ind}_{\mf p}^\G L(\mf l,\la)\right)=0$ for $i>0$. Finally, since $\mu=\la -\sum_{\alpha \in J}\alpha$ is a weight of $L({\mf l},\la)$ and $L(\mu)$ is a composition factor of ${\rm Ind}_{\mf p}^\G L(\mf l,\la)$ by Proposition \ref{prop:same:central}(1), we have $\mu=\la$. The proof completes.
\end{proof}

Let $\la\in\La^+$ and suppose that $\mf p$ is a parabolic subalgebra with Levi subalgebra $\mf l\cong\mf q(k)$ that contains every positive root atypical to $\la$. Recalling $\mf l(\la)$ and $\mf p(\la)$ from \eqref{llambda}, we conclude that $\mf l(\la)\subseteq\mf l$ and $\mf p(\la)\subseteq \mf p$.

Suppose that
\begin{align*}
\mc E^{\G,\mf l(\la)}\left({\rm Ind}_{\mf p(\la)}^\G W_\la\right)=\sum_{\mu}a_{\la\mu} L(\mu),\quad
\mc E^{\mf l,\mf l(\la)}\left( {\rm Ind}_{\mf{p}(\la)}^\mf l W_\la\right)=\sum_{\mu}a'_{\la\mu} L(\mf l,\mu).
\end{align*}
Since $\la\succcurlyeq\mu$ above, we observe that $\mf l$ also contains all positive root atypical to $\la$.
We compute, using Lemma \ref{lem:nested:Euler} and Proposition \ref{prop:typical1}
\begin{align*}
\sum_{\mu}a_{\la\mu}L(\mu)= & \mc E^{\G,\mf l(\la)}\left({\rm Ind}_{\mf p(\la)}^\G W_\la\right)= \mc E^{\G,\mf l}\left({\rm Ind}_{\mf p}^\G \mc E^{\mf l,\mf l(\la)}\left( {\rm Ind}_{\mf{p}(\la)}^\mf p W_\la\right)\right)\\
= & \mc E^{\G,\mf l}\left({\rm Ind}_{\mf p}^\G \mc \sum_{\mu}a'_{\la\mu}L(\mf l,\mu)\right)=\sum_{\mu}a'_{\la\mu} L(\mu).
\end{align*}
We conclude that
\begin{align}\label{aux:typ1}
a_{\la\mu}=a'_{\la\mu}.
\end{align}

The following is a crucial observation in the rank reduction algorithm in \cite[Theorem 3]{PS2}.

\begin{lem}\label{thm:PS:main}
Suppose that $\mf b\subseteq\mf p\subseteq\mf q\subseteq\G$ are parabolic subalgebras with Levi subalgebras $\mf h\subseteq\mf l\subseteq \mf k\subseteq\G$ respectively such that $\mf p\cap \mf k$ is the maximal parabolic subalgebra of $\mf k$ corresponding to removing the left-most node in the Dynkin diagram of $\mf k_{\bar 0}$. Then for $\la,\mu\in\La^+_\hZ$ and $i\ge 0$, we have
\begin{align*}
m^{\mf q,\mf p}_i(\la,\mu)= m^{\mf q,\mf p}_i\left(\la^\sharp,\mu^\sharp\right).
\end{align*}
\end{lem}

\begin{proof}
%Since it suffices to consider non-zero $m^{\mf q,\mf p}_i(\la,\mu)$ only, we may assume that $\la\succcurlyeq\mu$ and $\chi_\la=\chi_\mu$ so that $\la^\sharp\succcurlyeq\mu^\sharp$ by Lemma \ref{lem:partial order}.
In light of \eqref{aux:typ1} we only need to consider weights of the form that appear in \cite[Theorem 3]{PS2}.

We apply the reduction algorithm \cite[Theorem 3]{PS2} to both $m^{\mf q,\mf p}_i(\la,\mu)$ and $m^{\mf q,\mf p}_i\left(\la^\sharp,\mu^\sharp\right)$, and compare their values.
Note that in either reduction algorithm the case \cite[(1.11)]{PS2} cannot occur.
Applying \cite[Theorem 3]{PS2} once to $m^{\mf q,\mf p}_i(\la,\mu)$, we have
\begin{equation}\label{eq:reduction-1}
m^{\mf q,\mf p}_i(\la,\mu)= m^{\mf q',\mf p'}_{i'}(\la',\mu'),\\
\end{equation}
for some ${\mf p}'$ and ${\mf q}'$ with ranks smaller than those of ${\mf p}$ and ${\mf q}$, respectively, and some $i'\geq 0$. Moreover, applying \cite[Theorem 3]{PS2} to $m^{\mf q,\mf p}_i\left(\la^\sharp,\mu^\sharp\right)$ gives
\begin{equation}\label{eq:reduction-2}
m^{\mf q,\mf p}_i(\la^\sharp,\mu^\sharp)= m^{\mf q',\mf p'}_{i'}((\la')^\sharp,(\mu')^\sharp),
\end{equation}
with the same $\mf p'$, $\mf q'$, and $i'$ as in \eqref{eq:reduction-1}.
Now, we apply the reduction algorithm as far as possible so that $m^{\mf q',\mf p'}_{i'}(\la',\mu')$ corresponds to either \cite[Theorem 3(a)(1.8)]{PS2} or \cite[Theorem 4(c)]{PS2}, which is equivalent to the case when $m^{\mf q',\mf p'}_{i'}((\la')^\sharp,(\mu')^\sharp)$ corresponds to \cite[Theorem 3(a)(1.8)]{PS2} or \cite[Theorem 4(b)(1.16)]{PS2} with no zeroes in $(\la')^\sharp$ and $(\mu')^\sharp$, respectively. In this case, we can also see that
\begin{equation}\label{eq:reduction-3}
m^{\mf q',\mf p'}_{i'}(\la',\mu')= m^{\mf q',\mf p'}_{i'}((\la')^\sharp,(\mu')^\sharp)\in \{\,0,1\,\}.
\end{equation}
Therefore, it follows from \eqref{eq:reduction-1}--\eqref{eq:reduction-3} that
$m^{\mf q,\mf p}_i(\la,\mu)= m^{\mf q,\mf p}_i(\la^\sharp,\mu^\sharp)$.
\end{proof}

\begin{thm}\label{thm:same:mult}
For $\la,\mu\in\La^+_\hZ$, we have
\begin{align*}
a_{\la\mu}=a_{\la^\sharp\mu^\sharp}.
\end{align*}
\end{thm}

\begin{proof}
%We may assume that $\la\succcurlyeq\mu$ and $\chi_\la=\chi_\mu$ so that $\la^\sharp\succcurlyeq\mu^\sharp$ by Lemma \ref{lem:partial order}.
We fix the sequence of parabolic subalgebras of $\G$ of the form \eqref{seq:parabolic} such that, for each $s$, $\mf p_{s-1}\cap \mf l_{s}$ is the maximal parabolic subalgebra of $\mf l_s$ corresponding to the removal of the left-most node of the Dynkin diagram of $(\mf l_s)_{\bar 0}$. Now we apply Lemma \ref{thm:PS:main} to \eqref{formula:a}.
\end{proof}

\section{Brundan-Kazhdan-Lusztig theory}

Let us first briefly recall the results in \cite{Br2}.
Let $q$ be an indeterminate. Let ${\mscr U}=U_q(\mf b_\infty)$ be the quantum group associated to the Lie algebra of type $\mf b_\infty$, and let $\mscr V$ be its natural representation with basis $\{\,v_a\,|\,a\in\Z \,\}$. Fix $n\geq 1$. The space $\mscr T^n:=\mscr V^{\otimes n}$ is naturally a $\mscr U$-module with standard monomial basis $\{\,N_\la:=v_{\la_1}\otimes \cdots\otimes v_{\la_n}\,|\,\la=\sum_{i=1}^n\la_i\varepsilon_i\in\La_\Z\,\}$. For $\la\in \La_\Z$, set $M_\la:=(q+q^{-1})^{z(\la)}N_\la$, where $z(\la)=n-\ell(\la)$.
We can define a topological completion $\widehat{\mscr T}^{ n}$, compatible with the Bruhat ordering, on which one has a bar involution $-$. There exist unique bar-invariant topological bases called the {\em canonical} and {\em dual canonical basis}, denoted by $\{\,T_\la\,|\,\la\in\La_\Z\,\}$ and $\{\,L_\la\,|\,\la\in\La_\Z\,\}$, respectively \cite[Theorem 2.22]{Br2}. We have
\begin{align*}
T_\la=\sum_{\mu}t_{\mu\la}(q)N_\mu,\quad
L_\la=\sum_{\mu}\ell_{\mu\la}(q)M_\mu,
\end{align*}
for some polynomials $t_{\mu\la}(q)\in\Z[q]$ and $\ell_{\mu\la}(q)\in\Z[q^{-1}]$ such that $t_{\mu\la}(q)=\ell_{\mu\la}(q)=0$, unless $\la\succeq\mu$ and $t_{\la\la}(q)=\ell_{\la\la}(q)=1$.

 Let
\begin{align*}
\mscr F^n:=\bigwedge^n \mscr V
\end{align*}
be the $q$-deformed $n$th exterior power of $\mscr V$ of type $\mf b_\infty$ as in \cite[Section 3]{Br2}, which is a $\mscr U$-module. This exterior module was first constructed in \cite{JMO}. The space $\mscr F^n$ has a basis $\{\,F_\la:=v_{\la_n}\wedge \cdots\wedge v_{\la_1}\,|\,\la\in\La^+_\Z\,\}$, and it admits a topological completion $\widehat{\mscr F}^n$. The bar involution on $\widehat{\mscr T}^{n}$ induces a bar involution on $\widehat{\mscr F}^n$, and there exists a unique bar-invariant topological basis $\{\,U_\la\,|\,\la\in\La^+_\Z\,\}$ called the {\em canonical basis} \cite[Theorem 3.5]{Br2}, where we have
\begin{align*}
U_\la=\sum_{\mu\in\La^+_\Z} u_{\mu\la}(q) F_\mu,
\end{align*}
for some $u_{\mu\la}\in\mathbb \Z[q]$ such that $u_{\mu\la}(q)\not=0$ unless $\mu\succeq\la$ and $u_{\la\la}(q)=1$. If $\pi : \mscr T^n \rightarrow \mscr F^n$ is the canonical projection map, then we have $\pi(T_{w_0\la})=U_\la$ if $\la\in \La^+_\Z$, and $0$ otherwise, where $w_0$ is the longest element in $W$.

%\begin{rem}{\rm
%Indeed, the triangularity of $(t_{\mu\la}(q))$, $(\ell_{\mu\la}(q))$, and $(u_{\mu\la}(q))$ holds with respect to the Bruhat ordering in \cite[Section 2.3]{Br2}, say $\succcurlyeq'$. By \cite[Lemma 2.15]{Br2}, it is not difficult to see that $\la\succcurlyeq'\mu$ implies $\la\succcurlyeq\mu$ for $\la,\mu\in \La^+_\Z$.}
%\end{rem}

For $\la\in\La^{+}_\Z$ we define
\begin{align}\label{eq:basis-E}
E_\la:=\sum_{\mu\in\La^+_\Z,\, \la\succeq \mu}u_{-w_0\la,-w_0\mu}(q^{-1}) L_\mu,
\end{align}
and let
\begin{equation*}
\mscr E^n:=\sum_{\la\in\La^+_\Z}\mathbb{Q}(q) E_\la \subset \widehat{\mscr T}^n.
\end{equation*}
One can show \cite[Theorem 3.16]{Br2} that $\{\,L_\la\,|\,\la\in\La^+_\Z\,\}$ is the unique bar-invariant basis of $\mscr E^n$, and $L_\la=\sum_{\mu\in\La^+_\Z}\ell_{\mu\la}(q) E_\mu$ for $\la\in\La^+_\Z$. Put
\begin{equation}\label{eq:Z-form of E}
\begin{split}
&\mscr E^n_{\Z[q,q^{-1}]}:=\sum_{\la\in \La^+_\Z}\Z[q,q^{-1}]L_\la=\sum_{\la\in \La^+_\Z}\Z[q,q^{-1}]E_\la,\\
&\mscr E^n_{\Z}:=\Z\otimes_{\Z[q,q^{-1}]}\mscr E^n_{\Z[q,q^{-1}]},
\end{split}
\end{equation}
where $\Z$ is the right $\Z[q,q^{-1}]$-module with $q$ acting on $\Z$ as $1$.  Let $E_\la(1)=1\otimes E_\la$ and $L_\la(1)=1\otimes L_\la\in \mscr E^n_{\Z}$ for $\la\in \La^+_\Z$.
The following is the main result in \cite{Br2}:

\begin{thm}[Theorem 4.52 in \cite{Br2}]\label{thm:Brundan Main} Let {\rm $\Psi:K(\mc O^{\texttt{fin}}_\Z) \rightarrow \mscr E^{n}_\Z$} be the $\Z$-linear isomorphism defined by
\begin{align*}
[E (\la)]\longmapsto E_{\la}(1) \quad \left(\la\in\La^+_\Z\right).
\end{align*}
Then $\Psi\left([L(\la)]\right)={L}_{\la}(1)$, for all $\la\in\La^+_\Z$.
\end{thm}

%Let
%\begin{align*}
%\mc F^{n,\times}:=\sum_{\la\in\La^{+}_{\Z^\times}}\mathbb Q(q) F_\la\subseteq \mc F^n.
%\end{align*}

%\begin{lem}
%The space $\mc F^{n,\times}$ is bar-invariant. %In particular, $\{\,U_\la\,|\,\la\in\La^+_{\Z^\times}\,\}$ is a basis of $\mc F^{n,\times}$.
%\end{lem}

%\begin{proof}
%By \cite[Lemma 3.4(2)]{Br2} we have $\ov{F}_\la\in F_\la+\sum_{\mu\succcurlyeq\la}\Z[q,q^{-1}]F_\mu$, for any $\la\in\La^+_\Z$. Now if $\la\in\La^{+,\times}_\Z$ and $\mu\succcurlyeq\la$, then $\mu\in\La^{+}_{\Z^\times}$, which implies that $\ov{F}_\la \in \mc F^{n,\times}$.
%\end{proof}

Now we consider $K\left(\mc O^{\texttt{fin}}_\hZ\right)$. Let
\begin{align*}
\mscr E^{n,\times}:=\mscr E^n/\mscr E^{n,0},
\end{align*}
where $\mscr E^{n,0}$ is the subspace of $\mscr E^{n}$ spanned by $\{\,E_\la\,|\,\la\in\La^{+}_\Z\setminus\La^{+}_{\Z^\times}\,\}$.
%\begin{align*}
%^{n,0}&:=\sum_{\la\in\La^{+}_\Z\setminus\La^{+}_{\Z^\times}}\mathbb Q(q) E_\la\subseteq \mc E^n.
%\end{align*}

\begin{lem}\label{lem:bar invariance of E^{n,0}}
The space $\mscr E^{n,0}$ is bar-invariant. Hence the bar involution on $\mscr E^{n}$ induces a bar involution on $\mscr E^{n,\times}$.
\end{lem}

\begin{proof}
By \eqref{eq:basis-E}, we see that the space $\mscr E^{n,0}$ is spanned by $\{\,L_\la\,|\,\la\in\La^{+}_\Z\setminus\La^{+}_{\Z^\times}\,\}$, which proves the claim.
\end{proof}

For $\la\in \La^{+}_{\Z^\times}$, put
\begin{equation*}
{\bf E}_\la:=\pi (E_\la), \quad {\bf L}_\la:=\pi(L_\la),
\end{equation*}
where $\pi: \mscr E^n \rightarrow \mscr E^{n,\times}$ is the canonical projection.

\begin{prop}
The set $\left\{\,{\bf L}_\la\,|\,\la\in \La^{+}_{\Z^\times}\,\right\}$ is the unique basis of $\mscr E^{n,\times}$ such that $\overline{\bf L_\la}={\bf L}_\la$ and ${\bf L}_\la\in {\bf E}_\la +\sum_{\la\succeq \mu, \la\neq \mu}q^{-1}\Z[q^{-1}]{\bf E}_\mu$ for $\la\in \La^{+}_{\Z^\times}$.
\end{prop}
\begin{proof}
By Lemma \ref{lem:bar invariance of E^{n,0}} and the bar-invariance of $L_\la$, we immediately have $\overline{\bf L_\la}={\bf L}_\la$ and ${\bf L}_\la\in {\bf E}_\la +\sum_{\la\succeq \mu, \la\neq \mu}q^{-1}\Z[q^{-1}]{\bf E}_\mu$. The uniqueness follows from \cite[Lemma 24.2.1]{Lu}.
\end{proof}

Define $\mscr E^{n,\times}_{\Z[q,q^{-1}]}$ and $\mscr E^{n,\times}_\Z$ as in \eqref{eq:Z-form of E}, and put ${\bf E}_\la(1)=1\otimes {\bf E}_\la$, ${\bf L}_\la(1)=1\otimes {\bf L}_\la\in \mscr E^{n,\times}_\Z$ for $\la\in \La^+_{\Z^\times}$.
We can now translate Theorem \ref{thm:same:mult} into Brundan's Fock space language for the Kazhdan-Lusztig theory of $\mf q(n)$ \cite{Br2}.
%Then we have the following realization of $K\left(\mc O^{\texttt{fin}}_\hZ\right)$ as a quotient of $K\left(\mc O^{\texttt{fin}}_\Z\right)$.

\begin{thm}\label{thm:main} Let {\rm $\Psi:K\left(\mc O^{\texttt{fin}}_\hZ\right) \rightarrow \mscr E^{n,\times}_\Z$} be the $\Z$-linear isomorphism defined by
\begin{align*}
[E (\la)]\longmapsto {\bf E}_{\la^\sharp}(1) \quad \left(\la\in\La^+_\hZ\right).
\end{align*}
Then $\Psi\left([L(\la)]\right)={\bf L}_{\la^\sharp}(1)$, for all $\la\in\La^+_\hZ$.
\end{thm}

\begin{proof}
For $\la\in\La^+_\hZ$, let ${\bf L}'_{\la^\sharp}=\Psi\left([L(\la)]\right)$. By \eqref{E:in:L}, we have
${\bf E}_{\la^\sharp}(1)=\sum_{\mu}a_{\la\mu}{\bf L}'_{\mu^\sharp}$.
On the other hand, by \eqref{eq:basis-E} and  Theorem \ref{thm:Brundan Main}, we have
${\bf E}_{\la^\sharp}(1)=\sum_{\la^\sharp\succeq\mu^\sharp}a_{\la^\sharp\mu^\sharp}{\bf L}_{\mu^\sharp}(1)$.
Finally by Theorem \ref{thm:same:mult}, we get
\begin{equation*}
\sum_{\la^\sharp\succeq\mu^\sharp}a_{\la^\sharp\mu^\sharp}{\bf L}_{\mu^\sharp}(1)=\sum_{\la^\sharp\succeq\mu^\sharp}a_{\la^\sharp\mu^\sharp}{\bf L}'_{\mu^\sharp}.
\end{equation*}
Since there are only finitely many $\mu$'s such that $\la\succcurlyeq \mu$, we conclude by Corollary \ref{cor:succcurlyeq = succeq} and induction that ${\bf L}_{\mu^\sharp}(1)={\bf L}'_{\mu^\sharp}$ for all $\mu\in \La^+_{\hf+\Z}$. This completes the proof.
\end{proof}

\begin{rem}
Theorems \ref{thm:same:mult} and \ref{thm:main} suggest a connection between the categories $\mc O^{\texttt{fin}}_\hZ$ and $\mc O^{\texttt{fin}}_\Z$.
\end{rem}

For $\la\in \La$, we denote by $\la^+$ the unique weight in $\La^+$ which is $W$-conjugate to $\la$.
Let $\la=\sum_{i=1}^n\la_i\varepsilon_i\in\La$ be given. For $1\leq i<j\leq n$ with $\la_i+\la_j=0$, define
\begin{equation*}
{\texttt R}_{i,j}(\la):=\la+a(\varepsilon_i-\varepsilon_j),
\end{equation*}
where $a$ is the smallest positive integer such that $\la+a(\varepsilon_i-\varepsilon_j)$ and
all ${\texttt R}_{k,l}(\la)+a(\varepsilon_i-\varepsilon_j)$ for $1\leq k<i<j<l\leq n$ with $\la_k+\la_l=0$ are $W$-conjugate to weights in $\La^+$.

Let $\la\in \La^+_{\Z^\times}$ with the degree of atypicality $r$. Let $1\leq i_1<\ldots<i_r<j_r<\ldots<j_1\leq n$ such that $\la_{i_s}+\la_{j_s}=0$ for $1\leq s\leq r$. Note that $\la_{i_1}>\cdots>\la_{i_r}>0>\la_{j_r}>\cdots>\la_{j_1}$.
Following \cite[Section 3-f]{Br1} we define
\begin{equation}
\begin{split}
{\texttt R}_\theta(\la)  &:= \left( {\texttt R}_{i_1,j_1}^{\theta_1}\circ {\texttt R}_{i_2,j_2}^{\theta_2}\circ\cdots \circ {\texttt R}_{i_r,j_r}^{\theta_r}(\la) \right)^+,\\
{\texttt R}'_\theta(\la) &:= \left( {\texttt R}_{i_r,j_r}^{\theta_r}\circ {\texttt R}_{i_{r-1},j_{r-1}}^{\theta_{r-1}}\circ\cdots \circ {\texttt R}_{i_1,j_1}^{\theta_1}(\la) \right)^+,
\end{split}
\end{equation}
for $\theta=(\theta_1,\ldots,\theta_r)\in \Z_+^r$. We put $|\theta|=\sum_{i=1}\theta_i$ for $\theta=(\theta_1,\ldots,\theta_r)\in \Z_+^r$.

\begin{thm}\label{thm:char for half}
For $\la\in \La^+_{\hf+\Z}$ with the degree of atypicality $r$, we have
\begin{itemize}
\item[(1)] $[E(\la)]= \sum_{\mu} [L(\mu)]$, where the sum is over all $\mu\in \La^+_{\hf+\Z}$ such that {\rm $\la={\texttt R}_\theta(\mu)$} for some unique $\theta\in \{0,1\}^r$,

\item[(2)] $[L(\la)]= \sum_{\mu,\theta}(-1)^{|\theta|}[E(\mu)]$, where the sum is over all $\mu\in \La^+_{\hf+\Z}$ and $\theta\in \Z_+^r$ such that {\rm $\la={\texttt R}'_\theta(\mu)$}.
\end{itemize}
\end{thm}
\begin{proof}
{(1)} By \cite[Theorem 3.36]{Br2}, we have for $\alpha, \beta\in \La^+_{\Z^\times}$
\begin{equation}\label{eq:coeff_a}
a_{\alpha\beta}=
\begin{cases}
1, & \text{if $\alpha={\texttt R}_\theta(\beta)$ for some $\theta\in \{0,1\}^r$},\\
0, & \text{otherwise},
\end{cases}
\end{equation}
where $r$ is the degree of atypicality of $\alpha$.
Note that $a_{\alpha\beta}$ is non-zero only if $\alpha\succeq\beta$ and $\alpha, \beta\in \La^+_{\Z^\times}(p)$ for some $p$.
It is easy to see that $\la={\texttt R}_\theta(\mu)$ if and only if $\la^\sharp={\texttt R}_\theta(\mu^\sharp)$ for $\la, \mu\in \La^+_{\hf+\Z}$. Hence the formula follows from Theorem \ref{thm:main}.

{(2)} For $\alpha, \beta \in\La^+_{\Z^\times}$, let
\begin{equation}\label{eq:coeff_b}
b_{\alpha\beta}=\sum_{\theta}(-1)^{|\theta|},
\end{equation}
where the sum is over all $\theta\in \Z_+^r$ such that $\alpha={\texttt R}'_\theta(\beta)$, and $r$ is the degree of atypicality of $\alpha$ and $\beta$. Note that $b_{\alpha\beta}$ is non-zero only if $\alpha, \beta\in \La^+_{\Z^\times}(p)$ for some $p$.

Recall that for $f, g\in \Z^{p|q}_+$, one can define $a_{fg}$ and $b_{fg}$ as in \eqref{eq:coeff_a} and \eqref{eq:coeff_b}, where $\texttt{R}_\theta(g)$ and $\texttt{R}'_\theta(g)$ are given in \cite[Section 3-f]{Br1}. Since ${\texttt X}_\theta(\gamma^\flat)={\texttt X}_\theta(\gamma)^\flat$ for $\gamma\in \La^+_{\Z^\times}(p)$ and $\theta\in \Z_+^r$ where ${\texttt X}=\texttt{R}$ or $\texttt{R}'$, we have by Lemma \ref{lem:increasing property} $\alpha={\texttt X}_\theta(\beta)$ in $\La^+_{\Z^\times}$ if and only if $\alpha^\flat={\texttt X}_\theta(\beta^\flat)$ in $\Z^{p|q}_+$ for $\alpha, \beta\in \La^+_{\Z^\times}(p)$. This implies that
\begin{equation}\label{eq:a=a circ}
a_{\alpha\beta}=a_{\alpha^\flat \beta^\flat},\quad\quad b_{\alpha\beta}=b_{\alpha^\flat \beta^\flat}.
\end{equation}
By Proposition \ref{prop:succeq and succeq'}, we also have
${\texttt R}'_\theta(\beta) \succeq \beta$
since  ${\texttt R}'_\theta(\beta)^\flat={\texttt R}'_\theta(\beta^\flat) \succeq_a \beta^\flat$ for any $\theta\in \Z_+^r$ \cite[Lemma 2.5]{Br1}.
Hence, $b_{\alpha\beta}$ is non-zero only if $\alpha\succeq \beta$.

For $\nu, \mu\in \La^+_{\hf+\Z}$, $b_{\nu\mu}$ is defined in the same way as in \eqref{eq:coeff_b}. It is clear that $b_{\nu\mu}=b_{\nu^\sharp\mu^\sharp}$ since $\nu={\texttt R}'_\theta(\mu)$ if and only if $\nu^\sharp={\texttt R}'_\theta(\mu^\sharp)$. Now, for $\la, \mu\in \La^+_{\hf+\Z}$
\begin{equation}\label{eq:ab=id-3}
\begin{split}
\sum_{\nu}a_{\la \nu}b_{\nu \mu}
&=\sum_{\nu}a_{\la^\sharp\nu^\sharp}b_{\nu^\sharp\mu^\sharp}
\ \ \ \ \ \ \ \ \ \ \ \ \ \ \  \ \text{by Theorem \ref{thm:same:mult}} \\
&=\sum_{\la^\sharp\succeq \nu^\sharp\succeq\mu^\sharp}a_{\la^\sharp\nu^\sharp}b_{\nu^\sharp\mu^\sharp}
\ \ \ \ \ \ \ \ \ \   \text{by \eqref{eq:coeff_a} and \eqref{eq:coeff_b}},
\end{split}
\end{equation}
where $\nu\in \La^+_{\hf+\Z}$.
The sum in \eqref{eq:ab=id-3} is non-zero only when $\la^\sharp, \mu^\sharp \in \La^+_{\Z^\times}(p)$ for some $p$, and the sum is over $\nu^\sharp\in \La^+_{\Z^\times}(p)$ with $\la^\sharp\succeq \nu^\sharp\succeq\mu^\sharp$. In this case, we have {\allowdisplaybreaks
\begin{equation}\label{eq:ab=id-4}
\begin{split}
\sum_{\la^\sharp\succeq \nu^\sharp\succeq\mu^\sharp}a_{\la^\sharp\nu^\sharp}b_{\nu^\sharp\mu^\sharp}
&=\sum_{\la^\natural\succeq_a \nu^\natural\succeq_a \mu^\natural}a_{\la^\natural \nu^\natural}b_{\nu^\natural \mu^\natural}
\ \ \ \ \ \ \  \text{by \eqref{eq:a=a circ} and Proposition \ref{prop:succeq and succeq'}}\\
&=\sum_{\la^\natural \succeq_a h \succeq_a \mu^\natural} a_{\la^\natural h}b_{h \mu^\natural}
 \ \ \ \ \ \ \ \ \ \   \text{by Lemma \ref{lem:increasing property}} \\
&=\delta_{\la^\natural \mu^\natural}
\ \ \ \ \ \ \ \ \ \ \ \ \ \ \ \ \ \ \ \ \ \ \ \ \ \ \ \text{by \cite[Corollary 3.36]{Br1}},
\end{split}
\end{equation}
where $\nu\in \La^+_{\hf+\Z}$ and $h\in\Z^{p|q}_+$. }
By \eqref{eq:ab=id-3} and \eqref{eq:ab=id-4}, we have
$\sum_{\nu}a_{\la \nu}b_{\nu \mu}=\delta_{\la\mu}$. We conclude that $[L(\nu)]=\sum_{\mu}b_{\nu\mu}[E(\mu)]$. The proof completes.
\end{proof}

\begin{rem}
Let $\la\in \La^+_\hZ$ and let $\mc O^{\texttt{fin}}_{\hZ,\chi_\la}$ denote the subcategory of $\mc O^{\texttt{fin}}_{\hZ}$ of modules of central character $\chi_\la$. Let $p$ and $q$ be determined by $\la$ as before (see, e.g., \eqref{def:p:q}) and let $\mc O^{\texttt{fin}}_{p|q,\chi_\la}$ denote the subcategory of finite-dimensional $\gl(p|q)$-modules of $\gl(p|q)$-central character $\chi_\la$, where $\la$ here is regarded as a $\rho$-shifted weight of $\gl(p|q)$. Theorem \ref{thm:char for half}, together with \cite[Corollary 3.36]{Br1}, seems to indicate a connection between $\mc O^{\texttt{fin}}_{\hZ,\chi_\la}$ and $\mc O^{\texttt{fin}}_{p|q,\chi_\la}$.
\end{rem}

\section{{Kac-Wakimoto type character formulas}}\label{sec:KW:form}

In this section, we derive a closed-form character formula for special classes of finite-dimensional irreducible modules, which are similar to {\em Kac-Wakimoto formula} for classical Lie superalgebras \cite{KW} (see also \cite{CK,CHR,GK,SZ1}). The results in this section are motivated by \cite{SZ1}.

Let $\la=\sum_{i=1}^n\la_i\varepsilon_i\in \La^+_{\hf+\Z}$ with the degree of atypicality $r>0$. Let $1\leq i_1<\ldots<i_r<j_r<\ldots<j_1\leq n$ be unique indices such that $\la_{i_s}+\la_{j_s}=0$ for $1\leq s\leq r$.
We put $S_\la=\{\,\beta_s:=\varepsilon_{i_s}-\varepsilon_{j_s}\,|\,1\leq s\leq r\,\}\subset \Phi^+$, the set of positive roots atypical to $\la$, and define $\la^{\Uparrow}$ to be the weight obtained from $\la$ by replacing $\la_{i_s}$ with $\la_{i_1}$ and replacing $\la_{j_s}$ with $\la_{j_1}$ for $2\leq s \leq n$. Let $\mu, \nu\in \La_{\hf+\Z}$ such that $S_\mu=S_\nu=S_\la$. If $\mu\succcurlyeq \nu$ with $\mu=\nu+\sum_{s=1}^rc_s\beta_s$ for some $c_s\in\Z_+$, then we put $|\mu-\nu|=\sum_{s=1}^rc_s$.

Following \cite{SZ1}, we say that
\begin{itemize}
\item[(1)] $\la$ is {\em totally connected} if for each $1\leq s\leq r-1$ and $\la_{i_s}> t > \la_{i_{s+1}}$, there exists $1\leq i\leq n$ such that $|\la_{i}|=t$,

\item[(2)] $\la$ is {\em totally disconnected} if  for each $1\leq s\leq r-1$, there exists $\la_{i_s} > t > \la_{i_{s+1}}$ such that $|\la_{i}|\neq t$ for any $1\leq i\leq n$.
\end{itemize}
Then we have the following Kac-Wakimoto type character formulas for totally connected and disconnected weights.

\begin{thm}\label{thm:KW formula}
Let $\la\in \La^+_{\hf+\Z}$ with the degree of atypicality $r$.
\begin{itemize}
\item[(1)] If $\la$ is totally connected, then we have
\begin{equation*}
{\rm ch}L(\la) = \frac{(-1)^{|\la^\Uparrow-\la|}2^{\lceil n/2 \rceil}}{r!D}\sum_{w\in W}(-1)^{\ell(w)} w\left(\frac{e^{\la^\Uparrow}}{\prod_{\beta\in S_\la}(1+e^{-\beta})}\right).
\end{equation*}

\item[(2)] If $\la$ is totally disconnected, then we have
\begin{equation*}
{\rm ch}L(\la) = \frac{2^{\lceil n/2 \rceil}}{D}\sum_{w\in W}(-1)^{\ell(w)} w\left(\frac{e^{\la}}{\prod_{\beta\in S_\la}(1+e^{-\beta})}\right).
\end{equation*}
\end{itemize}
\end{thm}
\begin{proof} (1)
Let $\la=\sum_{i=1}^n\la_i\varepsilon_i \in \La^+_{\hf+\Z}$ be a totally connected weight.
%If $r=0$, then the formula follows immediately  from Theorem \ref{thm:char for half}(2) (cf.~\cite[(1.17)]{PS2}). So we may assume that $r\geq 1$.
Let $1\leq i_1<\ldots<i_r<j_r<\ldots<j_1\leq n$ be such that  $\la_{i_s}+\la_{j_s}=0$ for $1\leq s\leq r$, and $S_\la=\{\,\beta_s=\varepsilon_{i_s}-\varepsilon_{j_s}\,|\, 1\leq s\leq r\,\}$.

Let $\texttt{KW}(\la)$ denote the right-hand side of the equation in (1).
For $\mu\in \La_{\hf+\Z}$, put
$\sigma(\mu)=2^{\lceil n/2\rceil}D^{-1}\sum_{w\in W}(-1)^{\ell(w)}w\left(e^\mu \right)$,
which is equal to ${\rm ch} E(\mu)$ by \eqref{eq:Euler character} when $\mu\in \La^+_{\hf+\Z}$. Let
\begin{equation*}
\begin{split}
X&=\la^\Uparrow -\sum_{s=1}^r\Z_+\beta_s,\\
X_1&=\left\{\,\la^\Uparrow-\sum_{s=1}^rc_s\beta_s\in X \ \,\Bigg\vert\,c_s\leq \la_{i_1}-\hf \ \text{for all $1\leq s\leq r$}\,\right\},\\
X_2&=X\setminus X_1.
\end{split}
\end{equation*}
%where $d_s=\la_{i_s}-\hf$ for $1\leq s\leq r$.
Then we have
\begin{equation}\label{eq:KW}
\texttt{KW}(\la) =  \texttt{KW}_1(\la) + \texttt{KW}_2(\la),
\end{equation}
where for $i=1,2$
\begin{equation*}
\begin{split}
\texttt{KW}_i(\la) = \frac{(-1)^{|\la^\Uparrow-\la|}}{r!}
\sum_{\mu\in X_i}(-1)^{|\la^\Uparrow-\mu|}\sigma(\mu).
\end{split}
\end{equation*}

Suppose that $\la^\sharp\in \La^+_{\Z^\times}(p)$ for some $p$. Put $q=n-p$. Then we may identify $\la$ with a dominant integral weight $\la^\natural\in \Z^{p|q}$ for $\gl(p|q)$ with respect to the standard Borel subalgebra (cf.~\cite{Br1,SZ1}). Note that
\begin{equation}\label{eq:KW_1}
\texttt{KW}_1(\la)=\sum_{\mu\in\La^+_{\hf+\Z}}k_{\la\mu}\sigma(\mu),
\end{equation}
for some $k_{\la\mu}\in \mathbb{Q}$, and $\mu^\sharp\in \La^+_{\Z^\times}(p)$ for $\mu$ such that $k_{\la\mu}\neq 0$.

Now, we use the Kac-Wakimoto type character formula for the irreducible $\gl(p|q)$-module with highest weight $\la^\natural$ \cite[Corollary 4.13]{SZ1}. Indeed, we can check without difficulty that the coefficient $k_{\la\mu}$ in \eqref{eq:KW_1} coincides with the coefficient of the character of Kac module with highest weight $\mu^\natural$ in the Kac-Wakimoto type formula associated to $\la^\natural$ (see the right-hand side of (4.46) in \cite{SZ1}).
Therefore, by \cite[Corollary 3.39(ii) and (4.40)]{Br1} and Theorem \ref{thm:char for half}(2), we conclude that
\begin{equation}\label{eq:KW_1=irr}
{\rm ch}L(\la)=\texttt{KW}_1(\la).
\end{equation}

Next, for $\mu=\sum_{i=1}^n\mu_i\varepsilon_i\in X_2$, define $\mu^{\Uparrow_-}$ to be the weight obtained from $\mu$ by replacing all the negative (resp. positive) $\mu_i$ for $1\leq i\leq p$ (resp. $p<i\leq n$) with $-\hf$ (resp. $\hf$). Put
\begin{equation*}
\begin{split}
\overline{X_2}=\{\,\nu\,|\,\nu=\mu^{\Uparrow_-}\ \text{for some $\mu\in X_2$}\,\},
\end{split}
\end{equation*}
and
\begin{equation*}
I_\nu=\left\{\,\beta\,|\,\beta=\varepsilon_i-\varepsilon_j\in S_\la,\ \nu_i=-\nu_j=-1/2\,\right\},
\end{equation*}
for $\nu\in \overline{X_2}$.
Then we have {\allowdisplaybreaks
\begin{equation*}
\begin{split}
\texttt{KW}_2(\la)&=
\frac{(-1)^{|\la^\Uparrow-\la|}}{r!}
\sum_{\mu\in X_2}(-1)^{|\la^\Uparrow-\mu|}\sigma(\mu)\\
&=\frac{(-1)^{|\la^\Uparrow-\la|}}{r!}\sum_{\nu\in \overline{X_2}}(-1)^{|\la^\Uparrow-\nu|}\sum_{(k_\beta)_{\beta\in I_\nu}\in\Z_+^r}(-1)^{\sum k_\beta}\sigma\left(\nu-\sum_{\beta\in I_\nu}k_\beta\beta\right).
\end{split}
\end{equation*}}
For each $\nu=\sum_{i=1}^n\nu_i\varepsilon_i\in \overline{X_2}$, we have
\begin{equation}\label{eq:alter sum prod form}
\sum_{(k_\beta)_{\beta\in I_\nu}\in\Z_+^r}(-1)^{\sum k_\beta}\sigma\left(\nu-\sum_{\beta\in I_\nu}k_\beta\beta\right)=2^{\lceil n/2\rceil} D^{-1}\sum_{w\in W}(-1)^{\ell(w)}w\left(\frac{e^\nu}{\prod_{\beta\in I_\nu}(1+e^{-\beta})} \right).
\end{equation}
If we put $x_i=e^{\varepsilon_i}$ for $1\leq i\leq n$, then
\begin{equation}\label{eq:prod form}
\begin{split}
&\frac{e^\nu}{\prod_{\beta\in I_\nu}(1+e^{-\beta})}=
\frac{x_1^{\nu_1}\cdots x_n^{\nu_n}}{\prod_{\beta=\varepsilon_i-\varepsilon_j\in I_\nu}(1+x_i^{-1}x_j)}
=\frac{x_1^{\nu'_1}\cdots x_n^{\nu'_n}}{\prod_{\beta=\varepsilon_i-\varepsilon_j\in I_\nu}(x_i+x_j)},
\end{split}
\end{equation}
where $\nu'=\sum_{i=1}^n\nu'_i\varepsilon_i=\nu+ \sum_{\beta=\varepsilon_i-\varepsilon_j\in I_\nu}\varepsilon_i$. Since $\nu'_i=\nu'_j=1/2$ for $\beta=\varepsilon_i-\varepsilon_j\in I_\nu$, the product form in \eqref{eq:prod form} is invariant under the transposition $(i,j)\in W$, which in particular implies that the alternating sum \eqref{eq:alter sum prod form} is zero, and accordingly
\begin{equation}\label{eq:KW_2=0}
\texttt{KW}_2(\la)=0.
\end{equation}
Therefore, the formula follows from \eqref{eq:KW}, \eqref{eq:KW_1=irr}, and \eqref{eq:KW_2=0}.

(2) The proof is almost the same as in (1), where we replace $\la^\Uparrow$ with $\la$, and apply \cite[Corollary 4.15]{SZ1}. We leave the details to the reader.
\end{proof}

\end{document}